\documentclass[12pt]{amsart}

\usepackage[utf8]{inputenc}

\usepackage{amssymb,amsthm,amsmath,amsxtra,mathrsfs}
\usepackage[all]{xy}
\usepackage{fullpage}
\usepackage{comment}
\usepackage{hyperref}
\usepackage{color}
\usepackage{graphicx}
\usepackage{booktabs}
\usepackage{mathtools}

\numberwithin{equation}{subsection}

\theoremstyle{plain}
\newtheorem{thm}[equation]{Theorem}
\newtheorem{prop}[equation]{Proposition}
\newtheorem{lem}[equation]{Lemma}
\newtheorem{cor}[equation]{Corollary}

\newtheorem{ques}[equation]{Question}

\newtheorem*{cor*}{Corollary}
\newtheorem*{prob*}{Problem}
\newtheorem*{thm*}{Theorem}
\newtheorem*{thma*}{Theorem A}
\newtheorem*{thmb*}{Theorem B}
\newtheorem*{thmc*}{Theorem C}

\theoremstyle{definition}
\newtheorem{defn}[equation]{Definition}
\newtheorem{algm}[equation]{Method}

\theoremstyle{remark}
\newtheorem{exm}[equation]{Example}
\newtheorem{rmk}[equation]{Remark}

\newenvironment{enumalph}
{\begin{enumerate}}
{\end{enumerate}}

\newenvironment{enumalg}
{\begin{enumerate}}
{\end{enumerate}}

\DeclareMathOperator{\Aut}{Aut}

\DeclareMathOperator{\Frac}{Frac}

\DeclareMathOperator{\Gal}{Gal}

\DeclareMathOperator{\Jac}{Jac}

\DeclareMathOperator{\PGL}{PGL}

\DeclareMathOperator{\Spec}{Spec}
\DeclareMathOperator{\Supp}{Supp}
\DeclareMathOperator{\SL}{SL}

\DeclareMathOperator{\GL}{GL}

\newcommand{\C}{\mathbb C}
\newcommand{\F}{\mathbb F}
\newcommand{\G}{\mathbb G}
\def\AA{\mathbb A}
\def\PP{\mathbb P}
\newcommand{\Q}{\mathbb Q}
\newcommand{\R}{\mathbb R}

\newcommand{\Z}{\mathbb Z}
\newcommand{\Qbar}{\overline{\mathbb Q}}

\newcommand{\Belyi}{Bely\u{\i}}

\newcommand{\calL}{\mathcal{L}}

\newcommand{\calR}{\mathcal{R}}

\newcommand{\scrO}{\mathscr{O}}
\newcommand{\scrP}{\mathscr{P}}

\DeclareMathOperator{\opchar}{char}
\DeclareMathOperator{\et}{\mathrm{\acute{e}t}}

\newcommand{\Fbar}{{F}^{\mathrm{sep}}}

\newcommand{\Xbar}{\overline{X}}

\newcommand{\defi}{\textsf}

\def\phi{\varphi}
\def\rho{\varrho}

\def\ext{\!\mid\!}

\newcommand{\lexp}[2]{#1(#2)}

\usepackage[foot]{amsaddr}
\makeatletter
\renewcommand{\email}[2][]{%
  \ifx\emails\@empty\relax\else{\g@addto@macro\emails{,\space}}\fi%
  \g@addto@macro\emails{#2}
  \@ifnotempty{#1}{\g@addto@macro\emails{\textrm{ (#1)}}}
}
\makeatother

\begin{document}

\title{On explicit descent of marked curves and maps}

\author{Jeroen Sijsling}
\address[First author]{Mathematics Institute, Zeeman Building, University of
Warwick, Coventry CV4 7AL, UK; Department of Mathematics, Dartmouth College,
6188 Kemeny Hall, Hanover, NH 03755, USA; Universit\"at Ulm, Helmholtzstrasse
18, 89081 Ulm, Germany}
\email[corresponding author]{sijsling@gmail.com}

\author{John Voight}
\address[Second author]{Department of Mathematics, Dartmouth College, 6188
  Kemeny Hall,
Hanover, NH 03755, USA}
\email{jvoight@gmail.com}
\date{\today}

\begin{abstract}
  We revisit a statement of Birch that the field of moduli for a marked
  three-point ramified cover is a field of definition. Classical criteria due
  to D\`ebes and Emsalem can be used to prove this statement in the presence of
  a smooth point, and in fact these results imply more generally that a marked
  curve descends to its field of moduli. We give a constructive version of
  their results, based on an algebraic version of the notion of branches of a
  morphism and allowing us to extend the aforementioned results to the wildly
  ramified case.  Moreover, we give explicit counterexamples for singular
  curves.
\end{abstract}

\keywords{Descent; field of moduli; field of definition; marked curves; \Belyi\ maps}

\maketitle

\setcounter{tocdepth}{2}
{\small\tableofcontents}

In his \emph{Esquisse d'un Programme} \cite{Grothendieck}, Grothendieck
embarked on a study of $\Gal(\Qbar \ext \Q)$, the absolute Galois group of
$\Q$, via its action on a set with a geometric description: the set of finite
morphisms $f:X \to \PP^1$ of smooth projective curves over $\Qbar$ unramified
away from $\{0,1,\infty\}$, known as \defi{\Belyi\ maps}.
%\Belyi\ had proven \cite{Belyi,Belyi2} that a smooth projective curve over
%$\C$ (or alternatively, a compact Riemann surface) can be defined over $\Qbar$
%if and only if it admits a \Belyi\ map, and part of Grothendieck's fascination
%with \Belyi's theorem was a consequence of the simple combinatorial and
%topological characterization that follows from it, encoded in the language of
%\defi{dessins} (also called \defi{drawings}). Progress has been made on this
%program---in part led by the fecundity of certain examples---but in large part
%the Galois action on the set of \Belyi\ maps remains mysterious.
The heart of this program is to understand over what number fields a \Belyi\
map is defined and when two Galois-conjugate \Belyi\ maps are isomorphic.  One
of the more subtle aspects of this investigation is the issue of
\emph{descent}: whether or not a \Belyi\ map is defined over its field of
moduli, which is then necessarily a minimal field of definition. Some concrete
descent problems were studied by Couveignes \cite{CouveignesCRF} and
Couveignes--Granboulan \cite{CouveignesGranboulan}; a more general theoretical
approach to the subject of descent of (maps of) curves
%(and much more general than the case of \Belyi\ maps)
was developed by D\`ebes--Douai \cite{DebesDouai} and D\`ebes--Emsalem
\cite{DebesEmsalem}.

By the classical theory of Weil descent, a \Belyi\ map with trivial
automorphism group descends to its field of moduli. As such, one typically
tries to eliminate descent obstructions by a simple rigidification that
eliminates as many non-trivial automorphisms as possible.
%In many cases, a suitable rigidification of the \Belyi\ map will eliminate
%nontrivial automorphisms, thus ensuring that the field of moduli (of the
%rigidification) becomes a field of definition. For example, a classical result
%due to Coombes--Harbater \cite{CoombesHarbater} shows that \Belyi\ maps that
%are generically Galois permit descent to the field of moduli; covers that are
%Galois (in the sense that the corresponding extension of function fields is
%Galois) equipped with an action of their automorphism group descend, and an
%explicit descent algorithm has been exhibited by Sadi \cite{Sadi}.
In his article on dessins, Birch claims \cite[Theorem 2]{Birch}
%that the issue of descent can be overcome by an intuitive rigidification: he
%states
that by equipping a \Belyi\ map $f : X \to \PP^1$ with a point $P \in X(\Qbar)$
satisfying $f(P)=\infty$, the marked tuple $(X,f;P)$ descends to its field of
moduli \cite[Theorem 2]{Birch}. Birch provides several references to more
general descent results, but upon further reflection we could not see how these
implied his particular statement. Obtaining a proof or counterexample was not
merely of theoretical importance: indeed, in order to determine explicit
equations for \Belyi\ maps, one needs a handle on their field of definition
\cite[\S 7]{SV}. Moreover, if the descent obstruction indeed vanishes, then it
is desirable to have general methods to obtain an equation over the field of
moduli. The issue under consideration is therefore also very relevant for
practical purposes.

Fortunately, it turned out that general results of D\`ebes--Emsalem \cite[\S
5]{DebesEmsalem} imply that a marked projective curve $(X;P)$ descends to its
field of moduli as long as the marked point $P$ is \emph{smooth}, something
that holds by definition for a \Belyi\ map $f : X \to \PP^1$ (as $X$ itself is
required to be smooth). Their argument then shows equally well that Birch's
statement holds true: a marked \Belyi\ map $(X,f;P)$ descends to its field of
moduli.

%It turns is crucial in our argument that a marked point $P$ on the curve $X$ is
%itself singular; if it is regular, then indeed a descent of $(X;P)$ \emph{does}
%exist, and Birch's statement holds true: a marked \Belyi\ map (with smooth
%source $X$) indeed descends to its field of moduli. In short,
%\begin{center}
%  \emph{a marked dessin descends}.
%\end{center}
%One can show that this result follows from more general results of
%D\`ebes--Emsalem \cite[\S 5]{DebesEmsalem}.

In this paper, we revisit the results of D\`ebes--Emsalem \cite{DebesEmsalem}
from a slightly different point of view that is more explicit from the point of
view of Weil cocycles.  Our main theorem (Theorem \ref{thm:BirchCurve}) is as
follows.  We define a \defi{marked map} $(Y,f; Q_1,\dots,Q_n ; P_1,\dots,P_m)$
over $F$ to be a map $f : Y \to X$ of curves over $F$ equipped with distinct
points $Q_1,\dots,Q_n \in Y(F)$ and $P_1,\dots,P_m \in X(F)$ with $n,m \geq 0$.
(For conventions on curves and maps, see section 1.)

\begin{thma*}
 Let
  \begin{equation*}
    (Y, f : Y \to X; Q_1,\ldots,Q_n; P_1,\dots,P_m) = (Y,f;\calR)
  \end{equation*}
  be a marked map of curves over a separably closed field $F = \Fbar$
  such that $n \geq 1$ and at least one of the points $Q_1 , \ldots Q_n$ is
  smooth. Then $(Y,f;\calR)$ descends to its field of moduli.
\end{thma*}

Theorem A extends the results of D\`ebes--Emsalem \cite{DebesEmsalem}, where
the curve $X$ was assumed to be of genus at least $2$ and where moreover the
order of $\Aut (X)$ was assumed to be coprime to the characteristic of the base
field.

Our approach to proving Theorem A is essentially self-contained, and we recover
the result of Birch on descent of marked \Belyi\ maps (Theorem
\ref{thm:BirchBelyi}). We define the \emph{branches} of a morphism, replacing
the choice of a tangential base point with something canonical.  This approach
has several advantages. First, it provides a more unified treatment, extending
results to the wildly ramified (but separable) case.  Second, the notion is
more conducive to constructive applications: by using branches it becomes
straightforward to read off a Weil cocycle from the given rigidification,
without having to calculate the canonical model of $\Aut (X) \backslash X$.
From the point of view of a computational \emph{Esquisse} \cite{SV}, this gain
is quite important. Finally, our approach leads us to construct some explicit
counterexamples to descent in the case where the marked point $P$ is not
smooth, as follows.

\begin{thmb*} \label{thm:BirchSingular}
  There exists a marked curve $(X;P)$ over $\C$ with $X$ projective whose field
  of moduli for the extension $\C \ext \R$ is equal $\R$ but such that $(X;P)$
  does not descend to $\R$.  
  
  Similarly, there exists a marked map $(X,f:X \to \PP^1; P; 0,1,\infty)$ over
  $\C$ with $X$ projective and $f$ unramified outside $\{0,1,\infty\}$ whose
  field of moduli for the extension $\C \ext \R$ is equal $\R$ but that does
  not descend to $\R$.   
\end{thmb*}

Our paper is organized as follows.  In section \ref{sec:prelim}, we introduce
the relevant background and notions under consideration, and we give precise
statements of the main results of this article.
%We avoid saying anything more precise in this introduction as precision is
%essential in this subject; the main statements are Theorems
%\ref{thm:BirchCurve} and \ref{thm:BirchBelyi}.
After a review of the classical Weil descent criterion in section
\ref{sec:Weil}, we present our theory of branches in section \ref{sec:branch}.
We give some concrete examples in section \ref{sec:XP}; these include the
explicit descent of marked hyperelliptic curves to their field of moduli as
well as the descent a marked Klein quartic $(X;P)$ using branches. In section
\ref{sec:counter}, we conclude the paper by presenting our counterexamples in
the singular case.  In an appendix, we consider descent of marked Galois
\Belyi\ maps in genus $0$.

\subsubsection*{Acknowledgments}
We would like to thank Bryan Birch, Pierre D\`ebes, Michel Emsalem, Enric Nart,
Andrew Obus, and the anonymous referee for their valuable comments on previous
versions of this article.  The second author was supported by an NSF CAREER
Award (DMS-1151047).

\subsubsection*{Declaration on competing interests}
The authors confirm that neither of them has any competing interests.

\section{Background, notation, and statements of main results}\label{sec:prelim}

In this section, we set up the basic background and notation we will use in the
rest of the paper.

\subsection{Definitions and notation}

Let $F$ be a field with separable closure $\Fbar$ and absolute Galois group
$\Gamma_F = \Gal (\Fbar \ext F)$.  When $\opchar F = 0$, we will also write
$\Fbar=\overline{F}$.  A \defi{curve} $X$ over $F$ is a geometrically integral,
separated scheme of finite type over $F$ of dimension $1$.
%A curve is \defi{nice} if it is smooth, projective, and geometrically
%integral.
A \defi{map} $f :Y \to X$ of curves over $F$ is a nonconstant morphism defined
over $F$.  If $X,Y$ are projective, this implies that $f$ is finite, hence
proper.  We will assume throughout and without further mention that all maps
are \emph{generically separable}, so the extension of function fields $F(Y)
\ext F(X)$ is separable.

An \defi{isomorphism} of maps $f:Y \to X$ and $f':Y' \to X'$ of curves over $F$
is a pair $(\psi,\phi)$ of isomorphisms of curves $\psi :Y \xrightarrow{\sim}
Y'$ and $\phi :X \xrightarrow{\sim} X'$ over $F$ such that $\phi f = f' \psi$,
i.e., such that the diagram
\begin{equation}\label{diag:aut}
  \begin{split}
    \xymatrix{
      Y \ar[r]^{\psi} \ar[d]^{f} & Y' \ar[d]^{f'} \\
      X \ar[r]^{\phi} & X'
    }
  \end{split}
\end{equation}
commutes; if $f$ is isomorphic to $f'$ we write $f \xrightarrow{\sim} f'$.

The absolute Galois group $\Gamma_F$ acts naturally on the set of curves and
the set of maps between curves defined over $\Fbar$---in both cases, applying
the Galois action of the automorphism $\sigma$ comes down to applying $\sigma$
to the defining equations of an algebraic model of $X$ (resp.\ $f$) over
$\Fbar$. The conjugate of a curve $X$ over $\Fbar$ by an automorphism $\sigma
\in \Gamma_F$ is denoted by $\lexp{\sigma}{X}$, and similarly that of a map $f$
by $\lexp{\sigma}{f}$.

Given a curve $X$ over $F$ and an extension $K \ext F$ of fields, we will
denote by $X_K$ the base extension of $X$ to $K$.  We denote the group of of
automorphisms of $X_K$ over $K$ by $\Aut (X) (K)$. This latter convention
reflects the fact that $\Aut (X)$ is a scheme over $F$, the $K$-rational points
of which are the $K$-rational automorphisms of $X$ (by work of Grothendieck on
the representability of the Hilbert scheme and related functors
\cite[4.c]{GrothendieckIV}). As in the introduction, we occasionally write
$\Aut (X)$ for the group of points $\Aut (X) (\Fbar)$ when no confusion is
possible, and for curves $Y$ and $X$ over $K$, we will often speak of an
isomorphism between $Y$ and $X$ over $K$ when more precisely an isomorphism
between $Y_K$ and $X_K$ over $K$ is meant.

We call a map of curves $f : Y \to X$ over $F$ \defi{geometrically generically
Galois} if the group $G = \Aut_X(Y,f)(\Fbar)$ acts transitively on the fibers
of $f$, or equivalently if the extension of function fields $\Fbar(Y) \ext
\Fbar(X)$ is Galois.  (This terminology reflects the fact that in this version
the usual Galois torsor property is only required to hold on the generic
points, and that the elements of the group $G$ need only be defined over
$\Fbar$.)  In what follows, we abbreviate \emph{geometrically generically
Galois} to simply \defi{Galois}. If $f : Y \to X$ is a Galois cover, then it
can be identified with a quotient $q_H : Y \to H \backslash Y$ of $Y$ by a
finite $F$-rational subgroup $H$ of $G$. (This quotient by $G$ exists by work
of Grothendieck \cite[\S 6]{GrothendieckIII}.)

Given a Galois map $f : Y \to X$ branching over ${P \in X(\Fbar)}$, the
ramification indices of the points $Q \in f^{-1}(P)(\Fbar)$ are all equal, and
we call this common value the \defi{branch index} of the point $x$ on $X$.

We will now define certain rigidifications of the maps and curves considered
above.

%The main objects of study in this article are curves and maps equipped with
%extra data of points, and accordingly we make the following definitions.

\begin{defn}
  Let $n \in \Z_{\geq 0}$. An $n$-\defi{marked curve} (or $n$-\defi{pointed
  curve}) $(X;P_1,\dots,P_n)$ over $F$ is a curve $X$ equipped with distinct
  points $P_1,\dots,P_n \in X(F)$.  An \defi{isomorphism} of $n$-marked curves
  $(X;P_1,\dots,P_n) \xrightarrow{\sim} (X';P_1',\dots,P_n')$ over $F$ is an
  isomorphism $\phi : X \xrightarrow{\sim} X'$ over $F$ such that
  $\phi(P_i)=P_i'$ for $i=1,\dots,n$.
\end{defn}

We will often use the simpler terminology \defi{marked curve} for a $1$-marked
curve.

\begin{defn}
  For $n,m \in \Z_{\geq 0}$, an $n,m$-\defi{marked map} $(Y,f; Q_1,\dots,Q_n ;
  P_1,\dots,P_m)$ over $F$ is a map $f : Y \to X$ of curves over $F$ equipped
  with distinct points $Q_1,\dots,Q_n \in Y(F)$ and $P_1,\dots,P_m \in X(F)$.
  An \defi{isomorphism} between $n,m$-marked maps over $F$ is an isomorphism of
  maps $(\psi,\phi) : f \xrightarrow{\sim} f'$ over $F$ such that $\psi (Q_i) =
  Q_i'$ for $j=1,\dots,n$ and $\phi(P_i) = P_i'$ for $i=1,\dots,m$.
\end{defn}

In the definition of an $n,m$-marked map, no relationship between the points
$Q_j$ and $P_i$ under $f$ is assumed; moreover, the points $Q_j$ may or may not
be ramification points and the points $P_i$ may or may not be branch points.

We recover the case of marked curves from marked maps by taking $Y = X$ and
choosing $f : X \to X$ to be the identity. To avoid cumbersome notation, we
will sometimes abbreviate
\begin{equation*}
  (Y,f;Q_1,\dots,Q_n;P_1,\dots,P_m)=(Y,f;\calR)
\end{equation*}
and refer to $\calR$ as the \defi{rigidification data}. We use similar notation
$(X;\calR)$ for marked curves.

\subsection{\Belyi\ maps}

We will be especially interested in the following special class of maps.

\begin{defn}\label{def:Belyi}
  Suppose $\opchar F = 0$. A \defi{\Belyi\ map} over $F$ is a marked map
  \begin{equation*}
    (X,f : X \to \PP^1; - ; 0,1,\infty)
  \end{equation*}
  over $F$ such that $X$ is projective and $f$ is unramified outside
  $\{0,1,\infty\}$.  A \defi{marked} \Belyi\ map $(X,f:X \to
  \PP^1;P;0,1,\infty)$ over $F$ is a \Belyi\ map that is marked. A
  \defi{\Belyi\ map with a marked cusp} over $F$ is a marked \Belyi\ map $(X,f
  : X \to \PP^1; P ; 0,1,\infty)$ over $F$ such that $f(P) = \infty$; the
  \defi{width} of the marked point $P$ is the ramification index of $P$ over
  $\infty$.
\end{defn}

\begin{rmk}
  Replacing $f$ by $1/f$ or $1/(1-f)$, one may realize any marked \Belyi\ map
  whose marking lies above $0$ or $1$ as a \Belyi\ map with a marked cusp.
\end{rmk}

An isomorphism $(\psi,\phi)$ of \Belyi\ maps fixes the target $\PP^1$: since
the automorphism $\phi$ of $\PP^1$ has to fix each of the points $0,1,\infty$
by definition, $\phi$ is the identity.  When no confusion can result, we will
often abbreviate a (marked) \Belyi\ map by simply $f$.

\begin{rmk}
  Relaxing the notion of isomorphism of \Belyi\ maps by removing the marked
  points $0,1,\infty$ on $\PP^1$ typically leads one to consider covers of
  conics instead of $\PP^1$, as in work of van Hoeij--Vidunas \cite[\S\S
  3.3--3.4]{vHV}; we discuss this setup further in the appendix.
\end{rmk}

Let $f : X \to \PP^1$ be a \Belyi\ map of degree $d$ over $\Qbar \subset \C$.
Numbering the sheets of $f$ by $1,\dots,d$, we define the \defi{monodromy
group} of $f$ to be the group $G \leq S_d$ generated by the monodromy elements
$s_0,s_1,s_\infty$ of loops around $0,1,\infty$; the monodromy group is
well-defined up to conjugation in $S_d$, as are the conjugacy classes
$C_0,C_1,C_\infty$ in $S_d$ of the monodromy elements, specified by partitions
of $d$.  We organize this data as follows.

\begin{defn}
  The \defi{passport} of $f$ is the tuple $(g;G;C_0,C_1,C_\infty)$, where $g$
  is the geometric genus of $X$ and $G \leq S_d$ and $C_0,C_1,C_\infty$ are the
  monodromy group and $S_d$-conjugacy classes associated to $f$. The
  \defi{size} of a passport is the cardinality of the set of
  $\Qbar$-isomorphism classes of \Belyi\ maps with given genus $g$ and
  monodromy group $G$ generated by monodromy elements in the $S_d$-conjugacy
  classes $C_0,C_1,C_\infty$, respectively.
\end{defn}

For all $\sigma \in \Aut(\C)$, the \Belyi\ map $\lexp{\sigma}{f}$ has the same
ramification indices and monodromy group (up to conjugation) as $f$. Therefore
the index of the stabilizer of $f$ in $\Aut(\C)$ is bounded by the size of its
passport. The size of a passport is finite and can be computed by combinatorial
or group-theoretic means \cite[Theorem 7.2.1]{Serre}; thus the field of moduli
of $f$ over $\Qbar$ is a number field of degree at most the size of the
passport.

We analogously define a \defi{marked passport} associated to a \Belyi\ map with
a marked cusp to be the tuple $(g;G;C_0,C_1,C_\infty;\nu)$ where $\nu$ is the
width of the marked cusp, corresponding to a cycle in the partition associated
to $C_\infty$. Once more the field of moduli of a \Belyi\ map with a marked
cusp over $\Qbar$ is a number field of degree at most the size of its marked
passport.

Now let $K\ext F$ be a Galois field extension, let $X$ be a curve over $K$, and
let $\Gamma=\Gal(K \ext F)$. The \defi{field of moduli $M^K_F (X)$ of $X$ with
respect to $K \ext F$} is the subfield of $K$ fixed by the subgroup
\begin{equation}
  \left\{ \sigma \in \Gamma = \Gal(K \ext F) : \lexp{\sigma}{X} \simeq X
  \right\}.
\end{equation}
In a similar way, one defines the field of moduli $M^K_F (Y;f;\calR)$ of a
marked map.  In this paper we will usually only consider the case where $K =
\Fbar$, and we simply call $M(X) = M^K_F (X)$ the \defi{field of moduli} of
$X$, and similarly for maps and marked variants.

\subsection{Statements}

The central question that we consider in this article is the following.

\begin{ques} \label{ques:descentcond}
  What conditions ensure that a curve, map, or one of its marked
  variants is defined over its field of moduli?
\end{ques}

For brevity, when an object is defined over its field of moduli with respect to
a separable extension $K \ext F$, we will say that the object \defi{descends
(to its field of moduli)}.

%One subtle issue concerns the existence of wild automorphisms, which interfere
%with the Galois-theoretic nature of descent.  To address this issue, we make
%the following definition.
%
%\begin{defn}
%  Let \[ (Y, f : Y \to X; Q_1,\dots,Q_m; P_1,\dots,P_n) = (Y;\calR) \] be a
%  marked map of nice curves.  We say that $Q \in Y(K)$ is \defi{tame} (with
%  respect to $\calR$) if $\pi(Q)$ is tamely ramified in the quotient $\pi:Y \to
%  \Aut(Y;\calR) \backslash Y$.
%
%  We say $\calR$ is \defi{tame} if a marked point $Q_i$ of $\calR$ is tame with
%  respect to $\calR$.
%\end{defn}

We are now in a position to give a precise formulation of the results
motivating this paper.
% , namely the answer to Question \ref{ques:descentcond} for
% the case of a curve (resp.\ \Belyi\ map) with at least one marked point
% (resp.\ marked cusp).

\begin{thm}\label{thm:BirchCurve}
  Let
  \begin{equation*}
    (Y, f : Y \to X; Q_1,\ldots,Q_m; P_1,\dots,P_n) = (Y,f;\calR)
  \end{equation*}
  be a marked map of curves over a separably closed field $F = \Fbar$ with $m
  \geq 1$ such that at least one of the points $Q_1 , \ldots Q_m$ is smooth.
  Then $(Y,f;\calR)$ descends to its field of moduli.
\end{thm}

We will prove Theorem \ref{thm:BirchCurve} as a consequence of Theorems
\ref{thm:BirchExp} and \ref{thm:infiniteaut} (the latter dealing with the case
of genus zero); it is a special instance of a more general geometric result
that shows how to extract a Weil cocycle from the rigidification provided in
the theorem. By contrast, marked curves can fail to descend if the marked point
is singular: in section~\ref{sec:counter} we construct two explicit
counterexamples, one on a curve $X$ that descends to $\R$ and another on a
curve that does not.

Before proceeding, we deduce several important corollaries.

\begin{cor}
  Let $(X;P)$ be a curve with a smooth marked point over a separably closed
  field $F = \Fbar$. Then $(X;P)$ descends.
\end{cor}

Another corollary is the following theorem.

\begin{thm}\label{thm:BirchBelyi}
  A marked \Belyi\ map $(X,f:X \to \PP^1; P; 0,1,\infty)$ over $\Qbar$
  descends.
\end{thm}

Theorem \ref{thm:BirchBelyi} was claimed by Birch \cite[Theorem 2]{Birch} in
the special case of a \Belyi\ map with a marked cusp; but his proof is
incomplete. In work of D\`ebes--Emsalem \cite[\S 5]{DebesEmsalem}, a proof of
Theorem \ref{thm:BirchBelyi} is sketched using a suitable embedding in a field
of Puiseux series.  The following corollary of Theorem \ref{thm:BirchBelyi}
then follows immediately.

\begin{cor}\label{cor:markedmoduli}
  A \Belyi\ map $f:X \to \PP^1$ is defined over a number field of degree at
  most the minimum of the sizes of the marked passports $(X,f; P; 0,1,\infty)$,
  where $P \in f^{-1}(\{0,1,\infty\})$.
\end{cor}

Corollary \ref{cor:markedmoduli} often enables one to conclude that the \Belyi\
map itself descends, as the following result shows.

\begin{cor}\label{cor:fombound}
  Let $f:X \to \PP^1$ be a \Belyi\ map over $\Qbar$ with a marked cusp $P$
  whose ramification index is unique in its fiber $f^{-1}(\infty)$. Then the
  \Belyi\ map $(X,f;-;0,1,\infty)$ descends.
\end{cor}

\begin{proof}
  Let $F$ be the field of moduli of the marked map $(X,f;-;0,1,\infty)$ with
  respect to the extension $\Qbar \ext \Q$.  Let $\sigma \in \Gamma_F$ and
  consider the conjugate \Belyi\ map $(\lexp{\sigma}{f} ;-; 0,1,\infty)$ with
  source $\lexp{\sigma}{X}$. Since $F$ is the field of moduli, there exists an
  isomorphism
  \begin{equation*}
    (\psi_{\sigma},\phi_{\sigma}) :
    (\lexp{\sigma}{f} ;-; 0,1,\infty)
    \xrightarrow{\sim} (X,f ;-; 0,1,\infty)
  \end{equation*}
  with $\phi_{\sigma}$ the identity map (as noted after Definition
  \ref{def:Belyi}).  We see that
  \begin{equation}
    f (\psi_\sigma (\lexp{\sigma}{P})) = \lexp{\sigma}{f} (\lexp{\sigma}{P}) =
    \lexp{\sigma}{\infty} = \infty = f(P)
  \end{equation}
  so $\psi_\sigma (\lexp{\sigma}{P}) \in f^{-1}(\infty)$. The ramification
  index of $f = \lexp{\sigma}{f} \psi_{\sigma}^{-1}$ at $\psi_{\sigma}
  (\lexp{\sigma}{P})$ is equal to the ramification of $\lexp{\sigma}{f}$ at
  $\lexp{\sigma}{P}$, which in turn equals that of $f$ at $P$. By the
  uniqueness hypothesis, we must have that $(\psi_{\sigma},\phi_{\sigma})$
  sends $\lexp{\sigma}{P}$ to $P$.

  Since $\sigma$ was arbitrary, this means that the field of moduli of the
  marked \Belyi\ map $(X,f; P; 0,1,\infty)$ coincides with that of $(X,f;-;
  0,1,\infty)$. By Theorem \ref{thm:BirchBelyi}, $(X,f; P; 0,1,\infty)$
  descends to this common field of moduli and thus so does $(X,f;-;
  0,1,\infty)$.
\end{proof}

\begin{rmk}\label{rmk:fombound}
  The hypothesis of Corollary \ref{cor:fombound} is very often satisfied. When
  it is not, one can still try to obtain a model of a \Belyi\ map over a small
  degree extension of its field of moduli by ensuring that marking a point does
  not make the size of the passport grow too much; for example, if $P$ is a
  point of maximal ramification index then the automorphism group of the marked
  tuple $(X,f;P;0,1,\infty)$ may be of small index in that of
  $(X,f;-;0,1,\infty)$.
\end{rmk}

\section{Weil cocycles}\label{sec:Weil}

Our main tool for the construction of examples and counterexamples is the
\emph{Weil cocycle criterion}, which we will give in Theorem \ref{thm:Weil}.
For more details, we refer to Serre \cite[Ch.\ V, 20, Cor.2]{SerreAlg} and to
Huggins's thesis \cite{Huggins}, which is an excellent exposition on descent of
curves.

%Enlarging
%the base field $F$ if necessary, we may and do suppose that the field of moduli
%$M^K_F (X)$ of $X$ is equal to $F$. Thus, for $\sigma \in \Gamma=\Gal (K \ext
%F)$ there exist isomorphisms $\phi_{\sigma} : \lexp{\sigma}{X}
%\xrightarrow{\sim} X$ over $K$. Under these hypotheses Weil proved the
%following result.

\subsection{Weil cocycle criterion}

Throughout this section, we let $K \ext F$ be a (possibly infinite) Galois
extension, let $\Gamma = \Gal(K \ext F)$, and let $X$ be a curve over $K$ whose
field of moduli with respect to the extension $K \ext F$ equals $F$.

\begin{thm}[Weil cocycle criterion] \label{thm:Weil}
  The curve $X$ descends if and only if there exist isomorphisms
  \begin{equation} \label{eq:isoms}
    \{ \phi_{\sigma}:\lexp{\sigma}{X} \xrightarrow{\sim} X \}_{\sigma \in
    \Gamma}
  \end{equation}
  over $K$ such that
  the cocycle condition
  \begin{equation}\label{eq:WeilCoc}
    \text{$\phi_{\sigma \tau} = \phi_{\sigma} \lexp{\sigma}{\phi_{\tau}}$ for
    all $\sigma, \tau \in \Gamma$}
  \end{equation}
  holds.

  More precisely, if the isomorphisms \eqref{eq:isoms} satisfy
  \eqref{eq:WeilCoc}, then there exists a descent $X_0$ of $X$ to $F$ and an
  isomorphism $\phi_0 : X \xrightarrow{\sim} X_0$ over $K$ such that
  $\phi_{\sigma}$ is given as the coboundary
  \begin{equation}\label{eq:WeilCob}
    \phi_{\sigma} = \phi_0^{-1}\lexp{\sigma}{\phi_0}  .
  \end{equation}
\end{thm}

\begin{rmk}
  The Weil cocycle criterion can also be formulated for arbitrary (possibly transcendental) normal and
  separable extensions of $F$; in this case one has to add the condition that
  the subgroup $\{\sigma \in \Gamma : \lexp{\sigma}{X} = X \; \text{and} \;
  \phi_{\sigma} = \mathrm{id}_X \}$ has finite index in $\Gamma$.
\end{rmk}

\begin{rmk}
  We warn the reader not to confuse the descent $\phi_0$ with the isomorphisms
  $\phi_\sigma$ for $\sigma \in \Gamma$.
\end{rmk}

An important corollary of Theorem \ref{thm:Weil}, obtained by an immediate
uniqueness argument, is the following.

\begin{cor}\label{cor:trivautX}
  If $\Aut (X) (K)$ is trivial, then $X$ descends.
\end{cor}

%\begin{proof}
%  First, we claim that without loss of generality we may take $[K:F]<\infty$
%  and all isomorphisms $\phi_\sigma$ in \eqref{eq:isoms} defined over $K$.
%  Since $X$ is of finite type, it is defined over $K'$ where $F \subseteq K'
%  \subseteq K$ and $[K':F]<\infty$.  The set of all conjugates
%  $\lexp{\sigma}{X'}$ with $\sigma \in \Gal(K \ext F)$ is the same as the set
%  of conjugates $\lexp{\sigma'}{X}$ where $\sigma' \in \Gal (K' \ext F)$;
%  these conjugates are well-defined by our choice of $K'$. Replacing $X$ by
%  $X'$ and enlarging $K'$ if necessary, we can also suppose that the finite
%  number of isomorphisms $\lexp{\sigma}{f} : \lexp{\sigma}{X}
%  \xrightarrow{\sim} X$ are defined over $K'$.  Of course, if $\Aut(X)(K)$ is
%  trivial then so is $\Aut(X)(K')$.
%
%  The corollary then follows since both sides of \eqref{eq:WeilCoc} define an
%  isomorphism $\lexp{\sigma\tau}{X} \xrightarrow{\sim} X$ over $K$; their
%  compositional difference belongs to $\Aut(X)(K)$, which is trivial, so they
%  must be equal.
%\end{proof}

The Weil cocycle criterion is especially concrete when the base field $F$
equals $\R$. In this case we only have to find an isomorphism $\phi : \Xbar
\xrightarrow{\sim} X$ between $X$ and its complex conjugate $\Xbar$ and test
the single cocycle relation
\begin{equation}\label{eq:WeilCocR}
  \phi \overline{\phi} = 1
\end{equation}
that corresponds to the complex conjugation being an involution. Given any
$\phi$ as above, all other are of the form $\alpha \phi$, where $\alpha \in
\Aut (X)(\C)$.

\begin{rmk}
  More generally, if the extension $K \ext F$ is finite, there is a method to
  find all possible descents of $X$ with respect to the extension $K \ext F$;
  see Method \ref{algm:computedescent}.
\end{rmk}

In general, using Theorem \ref{thm:Weil} requires some finesse; when the
extension $K \ext F$ is infinite, it is not immediately clear through which
finite subextensions a descent cocycle could factor. However, in Corollary
\ref{cor:splitting} below we shall see that for marked curves or marked \Belyi\
maps it is possible to reduce considerations to an explicitly computable finite
subextension of $K \ext F$.

\subsection{Consequences}

Now let $(Y;\calR)$ be a marked curve or map, with $\calR$ the rigidification
data, and suppose again that the field of moduli of $(Y;\calR)$ with respect to
the extension $K \ext F$ equals $F$. In this case (as mentioned by
D\`ebes--Emsalem \cite[\S 5]{DebesEmsalem}), Theorem \ref{thm:Weil} still
applies to the marked curve $(Y;\calR)$ after replacing $\Aut(Y)$ by the
subgroup $\Aut(Y;\calR)$. Moreover, in case a Weil cocycle exists the marked
data then descend along with $Y$ to give a model $(Y_0;\calR_0)$ of $(Y;\calR)$
over $F$. We give a simple example of this general principle.

\begin{exm}\label{exm:P0gal}
  Let $(X;P)$ be a marked curve over $K$.  Suppose that $X$ has field of moduli
  $F$, with $\phi_{\sigma}:\sigma(X) \xrightarrow{\sim} X$ in
  \eqref{eq:WeilCoc} having the additional property that $\phi_{\sigma}
  (\lexp{\sigma}{P}) = P$.  Let $X_0$ be a descent of $X$ to $F$, with
  isomorphisms $\phi_0:X \xrightarrow{\sim} (X_0)_K$ over $K$. Let $P_0 =
  \phi_0(P) \in X_0(K)$. Then by \eqref{eq:WeilCob}, for all $\sigma \in
  \Gamma$ we have
  \begin{equation}\label{eq:PRat}
    \lexp{\sigma}{P_0} = \lexp{\sigma}{\phi_0} (\lexp{\sigma}{P}) = \phi_0
    \phi_{\sigma} (\sigma(P)) = \phi_0 (P) = P_0  .
  \end{equation}
  By Galois invariance, we see that $P_0 \in X_0(F)$. Therefore the marked
  curve $(X;P)$ descends.
\end{exm}

The following special case of Theorem \ref{thm:BirchCurve}, an analogue of
Corollary \ref{cor:trivautX}, is then clear from Example \ref{exm:P0gal}.

\begin{cor}\label{cor:trivaut}
  If the group $\Aut (Y,f;\calR)(K)$ is trivial, then $(Y,f;\calR)$ descends.
\end{cor}

\begin{cor}\label{cor:trivautfP}
  Suppose that $(Y,f;\calR)$ is given by
  \begin{equation*}
   (Y,f;Q_1,\dots,Q_n;P_1,\dots,P_m)=(Y,f;\calR)
  \end{equation*}
  with $n \geq 1$, and moreover suppose that $Q_1$ is \emph{not} a ramification
  point of $f$.  Then $(Y,f;\calR)$ descends.
\end{cor}

\begin{proof}
  We have $\Aut(Y,f;\calR)(K) \leq \Aut(Y,f)(K)$ and the latter group acts
  freely on those orbits that contains a non-ramifying point, so indeed
  $\Aut(Y,f;\calR)(K)$ is trivial.
\end{proof}

The moral is that when rigidification trivializes the automorphism group, then
the obstruction to descent vanishes.  At the other extreme, when the original
curve descends and the marked curve has the same automorphism group, then the
marked curve descends as well.

\begin{prop}\label{prop:auteq}
  Suppose that $\Aut(Y,f;\calR)(K) = \Aut (Y)(K)$. Let $F$ be the field of
  moduli of $(Y,f;\calR)$. Then $(Y,f;\calR)$ descends if $Y$ descends, and any
  descent datum for $Y$ gives rise to one for $(Y,f;\calR)$.
\end{prop}

\begin{proof}
  Let $Y_0$ be a descent of $Y$ to $F$ with isomorphism $\psi_0 : Y
  \xrightarrow{\sim} Y_0$ over $K$, and let $\sigma \in \Gal (K \ext F)$. We
  now interpret the map $f$ as a further rigidification of the curve $Y$, and
  correspondingly write $(Y;\calR')$ for $(Y,f;\calR)$. By hypothesis, there
  exists an isomorphism $\psi_{\sigma} : \lexp{\sigma}{Y} \xrightarrow{\sim} Y$
  over $K$ that sends the conjugate rigidification $\lexp{\sigma}{\calR'}$ to
  $\calR'$. Composing, we get an isomorphism
  \begin{equation}
    \psi'_{\sigma} : \lexp{\sigma}{Y_0}
    \xrightarrow{\lexp{\sigma}{\psi_0^{-1}}} \lexp{\sigma}{Y}
    \stackrel{\psi_{\sigma}}{\longrightarrow} Y
    \stackrel{\psi_0}{\longrightarrow} Y_0 .
  \end{equation}
  Moreover, if we let $\calR'_0 = \psi_0 (\calR')$, then the cocycle relation
  $\psi_{\sigma}=\psi_0^{-1} \lexp{\sigma}{\psi_0}$ shows that $\psi'_{\sigma}$
  sends $\sigma (\calR'_0)$ to $\calR'_0$: we have
  \begin{equation}
    \psi'_{\sigma} (\calR'_0) = \psi_0 (\psi_{\sigma}
    (\lexp{\sigma}{\psi_0^{-1}} (\psi_0 (\calR')))) = \psi_0 (\psi_{\sigma}
    (\psi_{\sigma}^{-1} (\calR'))) = \psi_0 (\calR') = \calR'_0 .
  \end{equation}
  Now we have assumed that $Y_0$ is defined over $F$. Therefore
  $\lexp{\sigma}{Y_0} = Y_0$, which shows that $\psi'_{\sigma}$ actually
  belongs to $\Aut(Y_0)(K)$. Since $\Aut(Y;\calR')(K) = \Aut (Y)(K)$ we equally
  well have $\Aut(Y_0;\calR'_0)(K) = \Aut (Y_0)(K)$. Hence in fact
  $\lexp{\sigma}{\calR'_0} = (\psi'_{\sigma})^{-1} (\calR'_0) = \calR'_0$.
  Since $\sigma$ was arbitrary, we see that the rigidification $\calR'_0$ is
  rational on $Y_0$, as desired.
\end{proof}

\begin{rmk}
  In case of proper inclusions $\{1\} \subsetneq \Aut (X;\calR)(K) \subsetneq
  \Aut (X)(K)$, it is possible that $X$ may descend to $F$ while $(X;\calR)$
  does not: we will see an example of this in Section \ref{sec:counter} which
  is minimal in the sense that the chain of inclusions
  \begin{equation*}
    \{1\} \subsetneq \Aut (X;\calR)(\C) \simeq \Z / 2 \Z \subsetneq \Aut (X)
    (\C) \simeq \Z / 4 \Z
  \end{equation*}
  is as small as possible.
\end{rmk}

\section{Branches}\label{sec:branch}

The geometric equivalent of a fundamental tool of D\`ebes--Emsalem
\cite{DebesEmsalem} is the consideration of what we will define as the
\emph{branches} of a morphism of curves over a point $P$. For a map tamely
ramified at $P$, branches can be interpreted by embeddings into certain rings
of Puiseux series. We revisit this definition in a more general geometric
context, which will extend to the wildly ramified case.

\subsection{Definitions}

Let $X$ be a curve over $F$, and let $P \in X (\Fbar)$ be a geometric point of
$X$. Let $\scrO_{X_{\et},P}$ be the local ring of $X$ at $P$ for the \'etale
topology. Using idempotents, we have a canonical decomposition
\begin{equation}\label{eq:odec}
  \scrO_{X_{\et},P} = \prod_i \scrO_{X_{\et},P,i}
\end{equation}
into a finite product of domains. Let $\scrP_{X_{\et},P,i}$ be the integral
closure of the domain $\scrO_{X_{\et},P,i}$ in
$\Frac(\scrO_{X_{\et},P,i})^{\mathrm{sep}}$, the separable closure of its
quotient field. We get a ring
\begin{equation}\label{eq:pdec}
  \scrP_{X_{\et},P} = \prod_i \scrP_{X_{\et},P,i}
\end{equation}
along with a canonical morphism $\Spec \scrP_{X_{\et},P} \to X$.

\begin{defn}\label{def:branch}
  Let $f : Y \to X$ be a map of curves over $F$. A \defi{branch} of $f$ over
  $P$ is a morphism $b:\Spec \scrP_{X_{\et},P} \to Y$ such that the following
  diagram commutes:
  \begin{equation} \label{diag:branch}
    \begin{aligned}
      \xymatrix
      {
        & Y \ar[d]^{f} \\
        \Spec \scrP_{X_{\et},P} \ar[r] \ar@{-->}[ur]^{\quad b} & X
      }
    \end{aligned}
  \end{equation}
  The set of branches of $f$ over $P$ is denoted $B (f,P)$.
\end{defn}

\begin{rmk}
  Geometrically, a branch of $f$ can be seen as an equivalence class of
  sections $V \to Y \stackrel{f}{\to} X$, where the composition $V \to X$ is a
  ``separable neighborhood'' of $P$ in the sense that it factors as $V \to U
  \to X$, where $V \to U$ is separable and where $U \to X$ is an étale
  neighborhood of $P$. Note that we could also have defined branches by using
  the total quotient ring of $\scrO_{X_{\et},x}$ instead of explicitly
  employing the factorization in the current definition. Another alternative at
  a smooth point is to use the completion of the usual local ring $\scrO_{X,
  x}$.
\end{rmk}

Definition \ref{def:branch} is most concrete when $P$ is a smooth point of $X$.
In this case $\scrP_{X_{\et},P}$ is the closure of $\scrO_{X_{\et},P}$ in
$\Frac(\scrO_{X,P})^{\mathrm{sep}}$. The ring $\scrO_{X_{\et},P}$ is the
subring of $\Fbar [[ t ]]$ consisting of power series that are algebraic over
the field of rational functions $\Fbar (t)$ in $t$ \cite[Prop.\ 4.10]{Milne}.
Therefore, if $\opchar F = 0$, then the elements of $\scrP_{X_{\et},P}$ can be
seen as elements of the ring of Puiseux series $\Fbar [[t^{1/\infty}]]$, or in
other words as elements of the field of Puiseux series $\Fbar ((t^{1/\infty}))$
whose monomials all have non-negative exponent. (This isomorphism explains our
notation for the ring $\scrP_{X_{\et},P}$.)

If $\opchar F = p$, then wild ramification can occur, as for example when
considering Artin--Schreier extensions; in this case the ring
$\scrP_{X_{\et},P}$ is no longer simply a subring of a field of Puiseux series,
and we have to use the field of \defi{generalized power series} $F ((t^{\Q}))$
instead. These are the formal linear combinations $\sum_{i \in \Q} c_i t^i$
whose support is a well-ordered subset of $\Q$. Kedlaya \cite[Theorem
10.4]{Kedlaya} has characterized the integral closure of $F (t)$ inside $F
((t^{\Q}))$ as those generalized power series that are \defi{$p$-automatic}; in
particular, branches will give rise to such $p$-automatic power series. A
practical illustration of how these generalized power series are obtained is
given in Example \ref{exm:genpow}.

\begin{rmk}
  By taking the fiber over the closed point of $\Spec \scrP_{X_{\et},P}$ (in
  the case of a valuation ring in a field of Puiseux series, we are setting
  $t=0$), we recover a lift of the point $P \in X(\Fbar)$. We can see $b$ as an
  infinitesimal thickening of this lift; the fact that we cannot use the local
  ring for the \'etale topology reflects that we need slightly thinner
  thickenings than that in this topology to obtain enough sections.
\end{rmk}

In general, the decomposition in \eqref{eq:odec} has more than one factor.
Using the categorical properties of products we see that giving a branch
amounts to specifying, for every $i$, a morphism $b_i$ making the following
diagram commute:
\begin{equation} \label{diag:branchi}
  \begin{aligned}
    \xymatrix
    {
      \Spec \scrP_{X_{\et},P,i} \ar@{-->}[r]^(.7){b_i} \ar[d] & Y \ar[d]^{f} \\
      \Spec \scrP_{X_{\et},P} \ar[r] & X
    }
  \end{aligned}
\end{equation}
In this way, we have effectively passed to the normalization of $X$. For the
identity map $X \to X$, the $b_i$ give what is classically called the set of
branches of $X$ at $P$.

\begin{rmk}
  Alternatively, one can first define branches in the case of a smooth point
  and then for a general morphism $f : Y \to X$ at a point $P$ to be a branch
  of the induced map $\widetilde{f} : \widetilde{Y} \to \widetilde{X}$ of
  normalizations at some point $\widetilde{P} \in \widetilde{X} (\Fbar)$ over
  $P$.
\end{rmk}

From here on, we will restrict to the case where $P \in X (\Fbar)$ is a smooth
point.

\subsection{Galois action and descent}

Given an element $\sigma$ of the absolute Galois group $\Gal (\Fbar \ext F)$,
there is a canonical map of sets
\begin{equation}\label{eq:Galact}
  \begin{split}
    B(f,P) & \to B(\lexp{\sigma}{f},\lexp{\sigma}{P}) \\
    b & \mapsto \lexp{\sigma}{b}
  \end{split}
\end{equation}
Here $\lexp{\sigma}{f}$ is the morphism $\lexp{\sigma}{Y} \to \lexp{\sigma}{X}$
induced by the action of $\sigma$ on the coefficients of $f$. The map on
branches is defined as follows. Given a branch $b : \Spec \scrP_{X_{\et},P} \to
Y$ of $f$ over $P$, we can conjugate it to obtain a morphism $\lexp{\sigma}{b}
: \Spec \lexp{\sigma}{\scrP_{X_{\et},P}} \to \lexp{\sigma}{Y}$.  We compose
with the canonical isomorphism
\begin{equation}
  \lexp{\sigma}{\scrP_{X_{\et},P}} \simeq
  \scrP_{\lexp{\sigma}{X}_{\et},\lexp{\sigma}{P}}
\end{equation}
% arising from the fact that both sides are abstractly isomorphic to the ring
% $\scrP_{X_{\et},x}$, and in both cases
where the $\Fbar$-algebra structure on this ring has been changed by
conjugation with $\sigma$.

In the upcoming proposition, we show that the set $B (f,P)$ has two very
pleasant properties, which should be thought of as generalizing the properties
of usual fibers over non-branch points (cf.\ Corollary \ref{cor:trivautfP}).
%The first of these properties is that the number of branches of $f : Y \to X$
%over a smooth point of $x$ equals the degree of $f$:

\begin{prop}\label{prop:constandfaithful}
  Let $f : Y \to X$ be a proper separable map of curves of degree $d$, and let
  $P \in X(\Fbar)$ be a smooth geometric point of $X$. Then the set $B (f, P)$
  has cardinality $d$, and the action of $\Aut (Y,f)$ on $B (f, P)$ is
  faithful.
\end{prop}

\begin{proof}
  Let $R = \scrO_{X_{\et},P}$ and $S=\scrP_{X_{\et},P}$ be as above. Since $P$
  is smooth, the rings $R$ and $S$ are integral domains; we denote their
  quotient fields by $K$ and $L$.  By definition, $L$ is a separable closure of
  $K$, and the point $P$ yields canonically a point $P_K \in X(K)$. Pulling
  back the map $f$ by $P_K$, we obtain another separable map of degree $d$,
  corresponding to a field extension of $K$.  By Galois theory, there are $d$
  points $Q_1 , \dots , Q_d$ in $Y (L)$ that lie over $P$.

  We claim that the points $Q_1,\dots,Q_d$ are the generic points of unique
  points in $Y (S)$. To see this, note that every point $Q_i$ is in fact
  defined over a finite subextension $K_i$ of $K$ contained in $L$. Let $R_i$
  be the integral closure of the strictly henselian discrete valuation ring $R$
  in $K_i$. Then $R_i$ is again a discrete valuation ring.  The morphism $f$ is
  proper because $Y$ and $X$ both are.  Therefore, by the valuative criterion
  of properness applied to the morphism $f$, the points $Q_i \in Y (K_i)$
  extend uniquely to points in $Y (R_i)$, proving our claim.

  To conclude, we show that the action of $\Aut(Y,f)$ is faithful: this follows
  since it is a specialization of the action of the automorphism group on the
  generic fiber.
\end{proof}

The following theorem can then be applied to give a constructive proof of
Theorem A (Theorem \ref{thm:BirchCurve}) in the case where the automorphism
group of the marked map is finite.  We finish the proof below.

\begin{thm}\label{thm:BirchExp}
  Let $(Y,f;\calR)$ be a marked curve over $\Fbar$ with finite automorphism
  group $G = \Aut (Y,f;\calR)(\Fbar)$. Let $\pi : Y \to G \backslash Y = W$ be
  the quotient map.
  \begin{enumalph}
    \item Let $Q \in Y (\Fbar)$ be smooth. Then $S = \pi (Q)$ is a smooth point
      of $W$, and the set of branches $B(\pi,S)$ is a torsor under $G$.
    \item Suppose that $\calR$ contains at least one smooth point $Q$ on $Y$.
      Then $(Y,f;\calR)$ descends.

      More precisely, let $S = \pi (Q)$, and suppose that $F$ is the field of
      moduli of $(Y,f;\calR)$.  Then given $b \in B(\pi,S)$, for all $\sigma
      \in \Gamma=\Gal(\Fbar \ext F)$ there exists a unique morphism
      $\phi_{\sigma}:\lexp{\sigma}{Y} \xrightarrow{\sim} Y$ such that
      $\lexp{\sigma}{b}$ is sent to $b$, and the collection $\left\{
      \phi_{\sigma} \right\}_{\sigma \in \Gamma}$ defines a Weil cocycle.
  \end{enumalph}
\end{thm}

\begin{proof}
  First, part (a).  The smoothness of the point $S$ follows from the fact that
  a quotient of the power series ring $\Fbar [[x]]$ by a finite group action is
  isomorphic to this same ring by the structure theorem for complete discrete
  valuation rings \cite[II, Th.\ 2]{SerreLF}. That the set $B(\pi,S)$ is a
  $G$-torsor then follows from Proposition \ref{prop:constandfaithful}, since
  the map $Y \to W$ is (generically) Galois. Note that the quotient morphism
  $\pi$ is separable and finite, hence proper.

  Now we prove (b). The existence of $\phi_{\sigma}$ follows because choosing
  just \emph{any} $\phi'_{\sigma} : \lexp{\sigma}{Y} \to Y$ maps
  $\lexp{\sigma}{b}$ to \emph{some} branch over $S$, which we can then modify
  to be $b$ by invoking the transitivity part of being a torsor. The
  $\phi_{\sigma}$ give a Weil cocycle because of the uniqueness part of being a
  torsor; both $\phi_{\sigma \tau}$ and $\phi_{\sigma}
  \lexp{\sigma}{\phi_{\tau}}$ send $\lexp{\sigma \tau}{b}$ to $b$ via pullback.
\end{proof}

\begin{rmk}\label{rmk:ratbranch}
  Theorem \ref{thm:BirchExp} in fact shows that there exists a descent
  $(Y_0;f_0;\calR_0)$ with a branch at the image $S_0$ of $S$ that is defined
  over $F$. Indeed, by a similar argument as that in \eqref{eq:PRat}, the
  coboundary corresponding to the Weil cocycle $\phi_{\sigma}$ maps $b$ to a
  branch defined over $F$ (in the sense that $b$ is sent to itself under the
  action \eqref{eq:Galact}).
\end{rmk}

The proof of Theorem \ref{thm:BirchExp} also shows how to obtain a finite
extension of $F$ over which a Weil cocycle can be constructed.

\begin{cor}\label{cor:splitting}
  Let $(Y,f;\calR)$ be a marked map with a smooth point $P$ as part of the
  rigidification data. Let $G$ be the automorphism group of $(Y,f;\calR)$. Then
  there exists a finite Galois extension $K \ext F$ over which a Weil cocycle
  for $(Y,f;\calR)$ can be constructed.
\end{cor}

\begin{proof}
  Let $K \ext F$ be a finite Galois extension over which all of $(Y,f;\calR)$,
  the elements of the group $G$, and the branches of the canonical morphism
  $\pi : Y \to G \backslash Y$ at $\pi (P)$ are defined.  Such an extension
  exists: defining expressions for $Y,f,\calR$ and the elements of the finite
  group $G$ have finitely many coefficients, and the finitely many branches are
  defined over a finite extension of $F$. Working over $K$ instead of $\Fbar$
  in Theorem \ref{thm:BirchExp} gives a descent relative to the extension $K
  \ext F$.
\end{proof}

\begin{rmk}
  As we will see explicitly later, computing the splitting field of the
  branches at a given point requires no more than determining the leading term
  of certain (generalized) power series whose order is known explicitly.
\end{rmk}

\begin{rmk}
  The counterexamples in Section \ref{sec:counter} all share the property that
  the number of branches over the points of the base curve is not constant,
  because of the merging of these branches over the singular points, which are
  (crucially) the only rational points of the base curve. This makes it
  impossible to read off a uniquely determined cocycle as in Theorem
  \ref{thm:BirchExp}.
\end{rmk}

\begin{rmk}
  As in Birch's original article \cite{Birch} we can consider the case where
  $X$ is a modular curve associated with a subgroup $\SL_2 (\Z)$, which we
  suppose to be defined over some number field $F$. Using a cusp of $X$ then
  allows one to use $q$-expansions with respect to a uniformizer $q$ of the
  appropriate width with respect to the cusp.

  Echelonizing a basis of modular forms gives a defining equation for $X$ over
  $F$; this does indeed turn out to lead to $F$-rational $q$-expansions, with
  the slight subtlety that one may need to twist $X$ to have a rational branch
  over $F$.
\end{rmk}

\subsection{Descent under infinite automorphism group}

To finish the proof of Theorem \ref{thm:BirchCurve}, it remains to deal with
the case where the automorphism group of the marked curve $(Y,f;\calR)$ is
infinite. In this case the map $f$ has to be the identity map, so we are
reduced to considering marked curves $(X;\calR) = (X;P_1, \ldots P_n)$ with
$n \geq 1$ for which one of the $P_i$, say $P_1$, is smooth.

\begin{lem}\label{lem:infg0}
  Let $(X;\calR) = (X;P_1, \ldots P_n)$ be a marked projective curve with $P_1$
  smooth and $\#\Aut(X;\calR)=\infty$. Let $\Sigma=X_{\textup{sing}}(\Fbar)$ be
  the set of singular points of $X$.  Then the following statements hold.
  \begin{enumalph}
    \item The geometric genus $g$ of $X$ is equal to $0$.
    \item We have $n + \#\Sigma \leq 2$, and the inverse image of $\calR \cup
      \Sigma$ with respect to the normalization $\widetilde{X} \to X$ has
      cardinality at most $2$.
  \end{enumalph}
\end{lem}

\begin{proof}
   An automorphism of $X$ extends to a unique automorphism of the normalization
   $\widetilde{X}$, so $\Aut(X;\calR) \leq \Aut(X) \leq \Aut(\widetilde{X})$.
   If $g \geq 2$ then $\Aut(\widetilde{X})$ is finite, so immediately we have
   $g=0,1$.

  Since $P_1$ is smooth, it lifts to a unique point $\widetilde{P}_1$ on
  $\widetilde{X}$, and any automorphism of $X$ fixing $P_1$ lifts to a unique
  automorphism of $\widetilde{X}$ fixing $\widetilde{P}_1$.  If $g=1$, then
  $(\widetilde{X},\widetilde{P}_1)$ has the structure of an elliptic curve so
  $\Aut(\widetilde{X},\widetilde{P}_1)$ is finite, and so $\Aut(X;\calR) \leq
  \Aut(X;P_1) \leq \Aut(\widetilde{X},\widetilde{P}_1)$.  So $g=0$, proving
  (a).

  Statement (b) then follows, since an automorphism of a curve of genus $0$ is
  determined by the image of three distinct points on this curve.
\end{proof}

\begin{thm} \label{thm:infiniteaut}
  Let $(X;\calR)$ be a marked curve over $\Fbar$ such that $\calR$ contains at
  least one smooth point on $X$.  Then $(X;\calR)$ descends.
\end{thm}

\begin{proof}
  Let $G=\Aut(X;\calR)$.  If $\#G<\infty$, we can apply Theorem
  \ref{thm:BirchExp}. If $\#G=\infty$, we apply Lemma \ref{lem:infg0} and have
  two cases.

  First suppose $X$ is projective and that $(X;\calR) = (X;P_1)$ for a smooth
  point $P_1$. By Lemma \ref{lem:infg0}(b), $X$ has at most one singular point.
  If $X$ is smooth, we can then take $(\PP^1; \infty)$ as a descent (cf.\ Lemma
  \ref{lem:g0caseXP}). This leaves the case where $X$ has a unique singular
  point $P_2$, and again by Lemma \ref{lem:infg0}(b), $P_2$ has a unique
  inverse image $\widetilde{P}_2$ on $\widetilde{X}$. The automorphism group
  $\Aut (X;\calR)$ is contained in the group of automorphisms $\Aut
  (\widetilde{X}; \widetilde{P}_1, \widetilde{P}_2) \simeq \G_m$ since $P_2$ is
  the unique singular point on $X$. Moreover, this inclusion is functorial.
  Therefore $\Aut (X;\calR)$ is isomorphic to a subvariety of $\G_m$, and since
  we assumed that it was infinite, it coincides with $\Aut (\widetilde{X};
  \widetilde{P}_1, \widetilde{P}_2)$.  Now $(\widetilde{X}; \widetilde{P}_1,
  \widetilde{P}_2)$ descends because $\widetilde{X}$ is smooth. Let $\left\{
  \phi_{\sigma} \right\}_{\sigma \in \Gamma}$ be a set of isomorphisms of $(X ;
  P_1, P_2)$ with its conjugates. By the property of the normalization, these
  lift to yield a set of isomorphisms $\left\{ \widetilde{\phi}_{\sigma}
  \right\}_{\sigma \in \Gamma}$ of $(\widetilde{X} ; \widetilde{P}_1,
  \widetilde{P}_2)$. Because $(\widetilde{X}; \widetilde{P}_1,
  \widetilde{P}_2)$ descends we can modify this latter set by automorphisms
  $\left\{ \widetilde{\alpha}_{\sigma} \right\}_{\sigma \in \Gamma}$, to obtain
  a Weil cocycle $\left\{ \widetilde{\phi}_{\sigma} \widetilde{\alpha}_{\sigma}
  \right\}_{\sigma \in \Gamma}$ for $(\widetilde{X}; \widetilde{P}_1,
  \widetilde{P}_2)$. Our statement on automorphism groups above means that the
  set of automorphisms $\left\{ \widetilde{\alpha}_{\sigma} \right\}_{\sigma
  \in \Gamma}$ yields a set of automorphisms $\left\{ \alpha_{\sigma}
  \right\}_{\sigma \in \Gamma}$, and by construction $\left\{ \phi_{\sigma}
  \alpha_{\sigma} \right\}_{\sigma \in \Gamma}$ is then a Weil cocycle for
  $(X;P_1,P_2)$.

  Now suppose that $X$ is projective with two marked points $(X;\calR) =
  (X;P_1,P_2)$. Again we are done if $X$ is smooth, since then a descent is
  given by $(\PP^1;\infty,0)$.  In the case where we admit singular points, we
  are reduced to the same case as that in the previous paragraph: $P_2$ is
  singular and has a unique inverse image $\widetilde{P}_2$ on $\widetilde{X}$.

  The case where $X$ is not projective follows by taking its smooth completion
  and applying the results above.
\end{proof}

\begin{rmk}
  It is indeed possible for the automorphism group $\Aut(X;P_1,P_2)$ to be
  finite and non-trivial when $X$ is singular and of (geometric) genus $0$.
  Take $\PP^1$ with affine coordinate $t$, pinch together the points $t = -1,1$
  to create a node, and let $X$ be the resulting singular curve.  Then
  $\Aut(X;0,\infty)$ is the subgroup of $\Aut(\PP^1;0,\infty)$ that preserves
  $\{-1,1\}$, and is therefore the subgroup $t \mapsto \pm t$. Examples where
  $P_1$ and $P_2$ are singular on $X$ can be constructed by constructing
  morphisms from suitable sublinear systems on $\PP^1$.
\end{rmk}

\subsection{Descent and the canonical model}

To conclude this section, we phrase our results in terms of a notion from
D\`ebes--Emsalem \cite{DebesEmsalem}. Let $(Y,f;\calR)$ be a marked map over
$\Fbar$ with field of moduli $F$ and \emph{finite} automorphism group $G = \Aut
(Y,f;\calR)$, and let $W = G \backslash Y$. Then by construction any choice of
isomorphisms \eqref{eq:isoms} gives rise to a Weil cocycle on $W$.

\begin{defn}\label{def:canmod}
  The \defi{canonical model} $W_0$ is the model of $W$ defined over $F$
  determined by the cocycle on $W$ induced by any choice of isomorphisms
  \eqref{eq:isoms}.
\end{defn}

The canonical model depends only on the isomorphism class of $(Y,f;\calR)$ over
$\Fbar$. In the canonical model, we denote the quotient morphism $\pi:Y \to G
\backslash Y$ and let $\psi_0 : W \xrightarrow{\sim} W_0$ be a coboundary
corresponding to the uniquely determined collection of isomorphisms $\{
\psi_{\sigma} \}_{\sigma \in \Gamma}$ induced by \eqref{eq:isoms}. Set $\pi_0
= \psi_0 \pi$.

\begin{prop}\label{prop:fomvsfod}
  Let $Q \in Y(\Fbar)$ be smooth. Then the field of moduli of $(Y,f;\calR;Q)$
  is the field of definition of the point $\pi_0 (Q)$ on $W_0$.
\end{prop}

\begin{proof}
  Let $F'=F(\pi_0(Q))$ be the field of definition of $\pi_0(Q)$ on $W_0$.

  First, we claim that $F'$ is contained in the field of moduli $M$ of
  $(Y,f;\calR;Q)$.  By Theorem \ref{thm:BirchCurve}, we may assume without loss
  of generality that $(Y,f;\calR;Q)=(Y_0;f_0;\calR_0;Q_0)$ is defined over $M$.
  Thus $Q \in Y(M)$.  This implies $\pi_0 (Q) \in W_0(M)$: we have $W_0 = G
  \backslash Y$ where $G=\Aut(Y,f;\calR)$ so that $\pi_0$ coincides with the
  natural projection $Y \to W = G \backslash Y$ defined over $M$. Thus $M
  \supseteq F'=F(\pi_0(Q))$.

  To conclude, we show $M \subseteq F'$. We then have to show that for all
  $\sigma \in \Gamma' = \Gal (\Fbar \ext F')$, there exist isomorphisms
  \begin{equation}\label{eq:isostronger}
    \phi_{\sigma} :
    (\lexp{\sigma}{Y};\lexp{\sigma}{f};\lexp{\sigma}{\calR};\lexp{\sigma}{Q})
    \xrightarrow{\sim} (Y,f;\calR;Q)
  \end{equation}
  so that $M$ is contained in $F'$. We are given that $(Y,f;\calR)$ has field
  of moduli $F$, so we have isomorphisms $\phi_{\sigma} : (\lexp{\sigma}{Y} ;
  \lexp{\sigma}{f} ; \lexp{\sigma}{\calR}) \xrightarrow{\sim} (Y,f;\calR)$ for
  all $\sigma \in \Gal(\Fbar \ext F)$.

  Let $\sigma \in \Gamma'$.  We will show that $\phi_{\sigma}$ can be chosen in
  such a way that additionally $\phi_{\sigma} (\lexp{\sigma}{Q}) = Q$.  Let
  $\psi_{\sigma} = \psi_0^{-1} \lexp{\sigma}{\psi_0}$. Then by construction of
  the canonical Weil cocycle for $W$ we know that $\pi \phi_{\sigma} =
  \psi_{\sigma} \lexp{\sigma}{\pi}$.  We obtain the following commutative
  diagram:
  \begin{equation} \label{eq:big_comm_diag}
    \xymatrix{
    \sigma(Y) \ar@/_4.5pc/[ddr]_{\sigma(\pi_0)} \ar[rr]^{\phi_\sigma} \ar[d]^{\sigma(\pi)} & & Y \ar[d]_{\pi} \ar@/^4.0pc/[ddl]^{\pi_0} \\
    \sigma(W) \ar[rr]^{\psi_\sigma} \ar[dr]_{\sigma(\psi_0)} & & W \ar[dl]^{\psi_0} \\
    & W_0
    }
  \end{equation}
  The rationality of $\pi_0(Q)$ over $F'$ implies that
  \begin{equation}\label{eq:Q0}
      \pi_0 (Q)
    = \lexp{\sigma}{\pi_0(Q)}
    = \lexp{\sigma}{\pi_0} (\lexp{\sigma}{Q}).
  \end{equation}
  We claim that $Q,\phi_{\sigma}(\sigma(Q)) \in Y(\Fbar)$ are in the same fiber
  of the map $\pi_0$.  Indeed, from \eqref{eq:Q0} and tracing through the
  diagram \eqref{eq:big_comm_diag}, we obtain
  \begin{equation}
      \pi_0(Q)
    = \lexp{\sigma}{\pi_0} (\lexp{\sigma}{Q})
    = \pi_0(\phi_{\sigma} (\lexp{\sigma}{Q}))  .
  \end{equation}
  Because $\Aut (Y,f;\calR)$ acts transitively on the fibers of $\pi$ and thus
  on those of $\pi_0$, we see that we can compose the chosen $\phi_{\sigma}$
  with an element of this group to obtain an isomorphism as in
  \eqref{eq:isostronger}.
\end{proof}

\begin{thm}\label{thm:twists}
  Let $Q \in Y(K)$ be smooth. Then $Q$ is $F$-rational on \emph{some} descent
  of $(Y,f;\calR)$ to $F$ if and only if $\pi_0 (Q)$ is $F$-rational.
\end{thm}

As a corollary, we obtain the following theorem, alluded to by D\`ebes--Emsalem
\cite[\S 5]{DebesEmsalem}.

\begin{thm}\label{thm:DebEms}
  Suppose that $W_0$ admits a smooth $F$-rational point. Then $(Y,f;\calR)$
  descends.
\end{thm}

\begin{proof}
  Let $Q \in Y(\Fbar)$ be such that $\pi_0(Q) \in V_0 (F)$.  Then the field of
  moduli of $(Y,f;\calR;Q)$ equals $F$ by Proposition \ref{prop:fomvsfod}. We
  get a descent $(Y_0;f_0;\calR_0;Q_0)$ of $(Y,f;\calR;Q)$ to $F$ by Theorem
  \ref{thm:BirchCurve}, which also gives a descent of $(Y,f;\calR)$ to $F$.
\end{proof}

\begin{rmk}
  Theorem \ref{thm:DebEms} applies in particular when $Y$ is not equipped with
  a rigidification. Consider for example the marked curve case $Y = X$ where
  $X$ is hyperelliptic, with hyperelliptic involution $\iota$; one can then
  show \cite{LRS} that a descent to $F$ induced by a point $P$ as in Theorem
  \ref{thm:twists} always exists, except possibly if $g(X)$ and $\# \Aut (X) /
  \langle \iota \rangle $ are both odd. That is, except in these special cases
  any descent can be obtained by marking a point on $X$. Note that the descents
  obtained will be of a rather special kind; indeed, Remark \ref{rmk:ratbranch}
  shows that these all have a rational branch over $Q$, and hence in particular
  a rational point.
\end{rmk}

\begin{rmk}
  The results in this section can also be obtained by using a \defi{tangential
  base point} (a germ of a regular function at a given point) as in Deligne
  \cite[\S 15]{Deligne} by splitting an exact sequence of fundamental groups,
  as explained by D\`ebes--Emsalem \cite{DebesEmsalem}. We prefer our new
  approach for two reasons. First, a choice of tangential base point is
  noncanonical, so it is cleaner (and more intuitive) to instead refine points
  into branches, which does not require such a choice. Second, the formalism of
  tangential base points has not been extended to the wildly ramified case to
  our knowledge, whereas our definition immediately covers this case. Phrased
  differently, we do not need to split the exact sequence of fundamental groups
  \cite{DebesEmsalem}, which is obtained only after restricting attention to
  prime-to-$p$ parts.
\end{rmk}

\section{Descent of marked curves}\label{sec:XP}

This section focuses on the explicit descent on marked curves, notably
hyperelliptic curves and plane quartics. Before treating these examples, we
indicate a general method for descending marked curves.

\subsection{Determining a Weil descent}

Let $K \ext F$ be a Galois extension, and let $(Y,f;\calR)$ be a marked map
over $K$. Suppose that the extension $K \ext F$ and the automorphism group
$\Aut (Y,f;\calR)$ are both finite. Then the following method finds all
possible descents of $(Y,f;\calR)$ with respect to the extension $K \ext F$.

\begin{algm}\label{algm:computedescent}
  This method takes as input a marked map $(Y,f;\calR)$ over $K$ and produces
  as output all descents $(Y_0,f_0;\calR_0)$ of $(Y,f;\calR)$ to $F$ up to
  isomorphism over $F$.

  \begin{enumalg}
    \item Compute a presentation $\Gamma = \langle \Sigma \mid \Pi \rangle$ of
      the Galois group $\Gamma = \Gal(K \ext F)$ in finitely many generators
      $\Sigma$ and relations $\Pi$.
    \item Compute $G = \Aut (Y,f;\calR)(K)$.
    \item For all $\sigma\in \Sigma$, compute an isomorphism $\psi_{\sigma} :
      (\lexp{\sigma}{Y},\lexp{\sigma}{f};\lexp{\sigma}{\calR})
      \xrightarrow{\sim} (Y,f;\calR)$ over $K$. If for one of these generators
      no such isomorphism exists, output the empty set and terminate.
    \item For all tuples $(\phi_{\sigma})_{\sigma \in \Sigma}$ in the product
      $\prod_{\sigma \in \Sigma} G \psi_{\sigma}$, use the relation
      \eqref{eq:WeilCoc} to determine the value of the corresponding cocycle on
      the relations in $\Pi$; retain a list $\calL$ of those tuples for which
      the corresponding cocycle is trivial on all elements of $\Pi$.
    \item For the tuples $(\phi_{\sigma})_{\sigma \in \Sigma}$ in $\calL$,
      construct a coboundary morphism $\phi_0 : (Y,f;\calR) \xrightarrow{\sim}
      (Y_0,f_0;\calR_0)$ to obtain a descent of $(Y,f;\calR)$.  Output this set
      of descents.
  \end{enumalg}
\end{algm}

\begin{rmk}
  Method \ref{algm:computedescent} applies to smoothly marked maps
  $(Y,f;\calR)$ over the separable closure $\Fbar$ to detect the existence of a
  descent: by Corollary \ref{cor:splitting}, we can always reduce the existence
  of a descent from $\Fbar$ to a finite extension $K \ext F$.  (There may be
  infinitely many descents of $(Y,f;\calR)$ over $\Fbar$, but only finitely
  many coming from a given finite extension $K \ext F$.)
\end{rmk}

\begin{rmk}
  In step 1 of the above method, one can work instead with the full group
  $\Gamma$ itself rather than generators and relations; then one loops over the
  elements of the group and checks compatibility with products.
\end{rmk}

The correctness of this method follows by the general approach in the previous
section. For now, we present only a `method' rather than a true algorithm, as
each of these steps may involve a rather difficult computer algebra problem in
general; however, efficient algorithms for finding isomorphisms exist for
smooth hyperelliptic or plane curves, making all but the final step of the
method algorithmic. The construction of a coboundary is somewhat more involved
but can be obtained in the aforementioned cases as well, essentially by finding
points on a certain variety over the ground field. We do not go into these
general problems here, as they form a substantial challenge of their own, but
we do indicate how to proceed in special cases below.

\begin{rmk}
  If $(Y,f;\calR)$ is already defined over $F$, the method above will find the
  \defi{twists} of $(Y,f;\calR)$ over $K$, that is, the set of $F$-isomorphism
  classes of marked maps curves over $F$ that become isomorphic to
  $(Y,f;\calR)$ over $K$.  In this case the isomorphisms $\phi_{\sigma}$ are
  themselves already in $G$, so this computation is slightly easier.
\end{rmk}

\subsection{Descent of marked smooth curves of genus at most
one}\label{subsec:lowgenus}

We begin by disposing of two easy cases in low genus.

\begin{lem}\label{lem:g0caseXP}
  Let $(X;P)$ be marked curve over $\Fbar$ with $X$ smooth of genus $\leq 1$.
  Then $(X;P)$ has field of moduli $F$.  If $g=0$, a descent to $F$ is given by
  $(X_0;P_0) = (\PP^1; \infty)$; if $g=1$, then the field of moduli is equal to
  $F(j(\Jac(X)))$.
\end{lem}

\begin{proof}
  The case $g=0$ is immediate.  So suppose $g=1$.  By translation on the genus
  $1$ curve $X$, we can take $(X_0;P_0) = (J ; \infty)$, where $J=\Jac(X)$ is
  the Jacobian of $X$ with origin $\infty$, which is defined over $F(j)$ where
  $j$ is the $j$-invariant of $J$.
\end{proof}

\subsection{Hyperelliptic curves}\label{subsec:hypcase}

We now pass to a class of examples that can be treated in some generality,
namely that of smooth hyperelliptic curves of genus at least $2$. Throughout,
we suppose that the base field $F$ does not have characteristic $2$.

Let $X$ be a smooth hyperelliptic curve over $\Fbar$ and let $\iota:X \to X$ be
the hyperelliptic involution.  Let $W=\langle \iota \rangle \backslash X$ and
$V = \Aut (X) \backslash X$. We get a sequence of natural projection maps
\begin{equation}\label{eqn:quotseq}
  X \stackrel{q}{\longrightarrow} W \stackrel{p}{\longrightarrow} V;
\end{equation}
the map $p$ is the quotient of $W$ by the \defi{reduced automorphism group} $G
= \Aut (X)(\Fbar) / \langle \iota \rangle$ of $W$ (or of $X$, by abuse).

We can identify $G$ with a subgroup of $\PGL_2 (\Fbar)=\Aut(W)(\Fbar)$. For
simplicity, we also suppose that \emph{$G$ does not contain any non-trivial
unipotent elements}.  This hypothesis implies that $G$ is isomorphic to one of
the finite groups $C_n, D_n, A_4, S_4$ and $A_5$.

The reduced automorphism group $G$ is also described as $G = \Aut(W;D)(\Fbar)$,
where $D$ is the branch divisor of $q$, or equivalently, the image on $W$ of
the divisor of Weierstrass points on $X$; thus $X$ can be recovered from
$(W;D)$ over $\Fbar$ by taking a degree $2$ cover of $W$ ramified over $D$.

\begin{rmk}
  Even if $(W;D) = (W_0;D_0)$ is defined over $F$, it is \emph{not} always
  possible to construct a cover $X_0$ of $(W_0;D_0)$ over $F$ whose branch
  divisor equals $D_0$.  Indeed, let $X$ be a hyperelliptic curve with field of
  moduli $F$ and $\Aut(X)(\Fbar)=\langle \iota \rangle$.  Then $W = V$, and
  there exists a canonical descent $W_0$ of $W$.  Moreover, by the same
  argument that we used to obtain (\ref{eq:PRat}), $D$ transforms to a
  $F$-rational divisor $D_0$. If the aforementioned cover could be constructed
  over $F$, then $X_0$ would be a descent of $X$.  But as is known classically
  (see also Remark \ref{rmk:ShimEarle}), there exist hyperelliptic curves of
  genus $2$ with generic automorphism group that do not descend.
\end{rmk}

Let $R \in X(\Fbar)$ and consider the marked curve $(X;R)$ over $\Fbar$. Let $Q
= q(R)$ be the image of $R$ on $W$, and similarly let $P = p(q(R))$ be its
image on $V$. Then we can also reconstruct $(X;R)$ from the datum $(W;Q;D)$
obtained by rigidifying $(W;Q)$ with the branch divisor $D$ of $q$.

\begin{prop}
  With the above notation, let $X'$ be a curve over $\Fbar$ obtained as a
  degree $2$ cover of $W$ ramifying over $D$, and let $R'$ be any preimage of
  $Q$ under the corresponding covering map. Then $(X';R')$ is isomorphic to
  $(X;R)$ over $\Fbar$.
\end{prop}

\begin{proof}
  This follows from the fact that a hyperelliptic curve over $\Fbar$ is
  uniquely determined by the branch locus of its hyperelliptic involution
  $\iota$, combined with the fact that $\iota$ acts transitively on the fibers
  of the quotient $X \to W$.
\end{proof}

We now suppose that $(X;R)$ has field of moduli $F$, which we can do without
loss of generality. By Theorem~\ref{thm:BirchCurve}, we know that $(X;R)$
descends to a marked curve $(X_0;R_0)$ over $F$, which upon taking quotients by
normal subgroups gives rise to a sequence of marked curves and morphisms
\begin{equation}\label{eqn:quotseq0}
  (X_0; R_0) \stackrel{p_0}{\longrightarrow} (W_0; Q_0)
  \stackrel{q_0}{\longrightarrow} (V_0; P_0)  .
\end{equation}
over $F$. In what follows, we will show how to calculate the descent $(X_0;
R_0)$ explicitly. Our methods do not yet need branches; for pointed
hyperelliptic curves, it suffices to use more elementary techniques involving
morphisms between curves of genus $0$.

We will first construct the descent $(V_0; P_0)$ of $(V, P)$. This is done as
in D\`ebes--Emsalem \cite{DebesEmsalem}; we start with any collection of
isomorphisms $\{ \phi_{\sigma} : \lexp{\sigma}{X} \to X \}_{\sigma \in
\Gamma_F}$ and consider the induced maps $\chi_{\sigma} : \lexp{\sigma}{V}
\xrightarrow{\sim} V$. By construction, the collection $\{ \chi_{\sigma}
\}_{\sigma \in \Gamma_F}$ satisfies the Weil cocycle relation. Therefore there
exists a map $\chi_0 : V \xrightarrow{\sim} V_0$ to a descent $V_0$ of $V$ over
$F$ such that $\chi_{\sigma} = \chi_0^{-1}\lexp{\sigma}{\chi_0}$. By the same
argument that we used to obtain (\ref{eq:PRat}), we see that $P_0 = \chi_0 (P)$
becomes an $F$-rational point on $V_0$.

The isomorphisms $\phi_{\sigma} : \lexp{\sigma}{X} \xrightarrow{\sim} X$ can be
explicitly calculated \cite{ANTS}. After determining the induced maps
$\chi_{\sigma} : \lexp{\sigma}{V} \xrightarrow{\sim} V$ between genus $0$
curves, we can explicitly calculate the map $\chi_0$, and with it the point
$P_0$ \cite{HidalgoReyes,Sadi}. We therefore assume the descent $\chi_0 : (V;
P) \xrightarrow{\sim} (V_0; P_0)$ to be given, and will now discuss how to
reconstruct a descent $(X_0; R_0)$ of $(X; R)$ from it.

In one case, determining $(X_0; R_0)$ turns out to be particularly easy.
Suppose that the reduced automorphism group of the pair $(X, R)$ is trivial, so
that $(W_0; Q_0) = (V_0; P_0)$. Then the image $D_0$ of the branch divisor $D$
on $V_0$ is defined over $F$. Since $V_0$ admits the point $P_0$ over $F$, it
is isomorphic to $\PP^1$ over $F$. Let $V_0 \xrightarrow{\sim} \PP^1$ be any
isomorphism over $F$ that sends $P_0$ to $\infty$, and let $p_0$ be a monic
polynomial cutting out the image of $\Supp (D_0) \backslash \infty$. Let
\begin{equation} \label{eq:X0desc}
  X_0 : y^2 = p_0 (x)
\end{equation}
If $R$ is not a Weierstrass point of $X$, then $\infty \notin \Supp (D_0)$, so
that $p_0$ is of even degree. In this case we let $R_0$ be the point at
infinity corresponding to the image of $(0,1)$ under the change of coordinates
$(x,y) \leftarrow (1/x,y/x^{g+1})$.  On the other hand, if $R$ is a Weierstrass
point of $X$, then $\infty \in \Supp (D_0)$ and $p_0$ is of odd degree, in
which case we let $R_0$ be the point corresponding to $(0,0)$ under the
aforementioned change of coordinates.

The following proposition is then immediate.

\begin{prop}\label{prop:VtoX}
  Let $(X;R)$ be a marked hyperelliptic curve over $\Fbar$ whose reduced
  automorphism group is trivial. Then the pair $(X_0; R_0)$ defined in
  \eqref{eq:X0desc} is a descent of $(X; R)$.
\end{prop}

\begin{rmk}
  In Proposition \ref{prop:VtoX}, the indicated point $R_0$ on $X_0$ admits a
  rational branch. In the case when $R$ is a not a Weierstrass point of $X$,
  this is because $R_0$ is unramified, whereas in the latter case we get a
  rational branch for the map $(x,y) \mapsto x$. This follows from a local
  power series expansion at infinity; since $y^2 = x + O(x^2)$, we get the
  rational branch corresponding to the root $y = \sqrt{x} + O(x)$.

  The quadratic twist defined by $c_0 y^2 = p_0 (x)$ also gives a descent of
  the marked curve $(X;R)$; but if $c_0 \in F^{\times} \setminus F^{\times 2}$,
  the map $(x,y) \to x$ does not admit a rational branch at the point at
  infinity, rather the morphism $(x,y) \mapsto c_0 x$ gives a rational branch.
  In general, the passage from points to branches gives rise to further
  cohomological considerations, and any branch on a descended curve $(X_0;
  R_0)$ for the quotient map $\Aut (X_0; R_0) \backslash (X_0; R_0)$ becomes
  rational on some uniquely determined twist of $(X_0; R_0)$.
\end{rmk}

Now suppose that the pair $(X;R)$ has nontrivial reduced automorphism group $G
= \Aut (X) / \langle \iota \rangle$. This automorphism group is then cyclic, of
order $n=\# G$. We will now show how to obtain a descent of $(X;R)$ in this
case.

Let $\alpha$ be a generator of $G$, and let $Q = q (R)$ be the image of $R$ on
the quotient $W = \langle \iota \rangle \backslash X$. Since we assumed that
$\alpha$ was not unipotent, we can identify $W$ with a projective line $\PP^1$
with affine coordinate $x$ in such a way that $Q$ corresponds to the point
$\infty \in \PP^1 (\Fbar)$ and such that $\alpha$ is defined by
\begin{equation}
  \alpha(x) = \zeta_n x
\end{equation}
with $\zeta_n$ a primitive $n$th root of unity in $\Fbar$.  Let $p$ be the
polynomial defining $X$ as a cover of $W$. Then because of our normalizations,
we either have
\begin{equation}
  p (x) = \pi (x^n)
\end{equation}
or
\begin{equation}
  p (x) = x \pi (x^n) .
\end{equation}
for some polynomial $\pi$. In the latter case, there is a single point at
infinity that is a Weierstrass point of $X$, whereas in the former there are
two ordinary points at infinity.

We normalize $\pi$ to be monic by applying a scaling in $x$ (which does not
affect the above assumptions on $\alpha$ and $Q$) and denote its degree by $d$.
By our assumption that the reduced automorphism group is of order
$n$, if we write
\begin{equation}
  \pi (x) = \sum_{i=0}^{d} \pi_{i} x^{d-i} = x^d + \pi_1 x^{d-1} + \pi_2
  x^{d-2} + \dots + \pi_{d-1} x + \pi_d
\end{equation}
then the subgroup of $\Z$ generated by $\{ i : \pi_i \neq 0 \}$ equals $\Z$.

Given a monic polynomial $\pi$ as above, we define the hyperelliptic curves
$X_{\pi}$ and $X'_{\pi}$ by
\begin{equation}
  X_{\pi} : y^2 = \pi (x)
\end{equation}
and
\begin{equation}
  X'_{\pi} : y^2 = x \pi (x) .
\end{equation}
The former curve admits a distinguished point $R_{\pi} = (0:1)$ at infinity
after the coordinate change $(x,y) \leftarrow (1/x,y/x^{g+1})$, whereas the
latter admits a Weierstrass point at infinity (which corresponds to $R'_{\pi} =
(0,0)$ under the same coordinate transformation).

Our only freedom left to modify the coefficients of $\pi$ is scaling in the
$x$-coordinate: this induces an action of $\Fbar$ on the coefficients $\pi_i$,
given by
\begin{equation}
  \alpha (\pi_1 , \pi_2, \ldots , \pi_d) = (\alpha^1 \pi_1 , \alpha^2 \pi_2 ,
  \ldots , \alpha^d \pi_d).
\end{equation}
We can interpret the coefficients of $\pi$ as a representative of a point in
the corresponding weighted projective space. Now there is a notion of a
\defi{normalized representative} of such a point \cite{LR}. We call a
polynomial $\pi$ \defi{normalized} if its tuple of coefficients is normalized.

The uniqueness of the normalized representative then shows the following
theorem.
\begin{thm}
  Let $(X;R)$ be a marked hyperelliptic curve over $\Fbar$ whose reduced
  automorphism group is tamely cyclic of order $n$. Then $(X;R)$ admits a
  unique model of the form $(X_{\pi}, R_{\pi})$ or $(X'_{\pi}, R'_{\pi})$, with
  $\pi$ normalized, depending on whether $R$ is a Weierstrass point of $X$ or
  not.
\end{thm}

\begin{proof}
  It is clear from the above that $(X;R)$ admits a model of the indicated form.
  On the other hand, let $\sigma \in G_F$ and consider the polynomial
  $\lexp{\sigma}{\pi}$ obtained by conjugating the coefficients of $\pi$ by
  $\sigma$. If the coefficients of $\pi$ were not stable under Galois, then we
  would find two distinct normalized representatives of the same weighed point,
  a contradiction. Therefore $\lexp{\sigma}{\pi} = \pi$, and since $\sigma$ was
  arbitrary, our theorem is proved.
\end{proof}

\begin{rmk}
  If we do not happen to know the field of moduli of the pair $(X;R)$, then the
  approach above calculates it as well. Indeed, $M^{\Fbar}_F (X;R)$ is
  generated by the coordinates of the normalized representative $c_{I,0}$. A
  similar remark applies to the upcoming Proposition \ref{prop:descg0}.
\end{rmk}

%\begin{rmk}
%  By normalizing as above, one in fact obtains a \emph{representative family}
%  of marked hyperelliptic curves.
%\end{rmk}

\subsection{The Klein quartic}

We now apply the method of branches to a plane quartic, namely the \defi{Klein
quartic}, defined by the equation
\begin{equation}\label{eq:KleinClassic}
  x^3 y + y^3 z + z^3 x = 0.
\end{equation}
By classical results (see the complete exposition by Elkies \cite{Elkies}),
this curve admits \Belyi\ map of degree $168$ that realizes the quotient map by
its full automorphism group. In what follows, we will determine a model of the
Klein quartic that admits a rational ramification point for this morphism that
is of index $2$. The geometry of the Klein quartic is exceedingly
well-understood due to its moduli interpretation. As such, our result is not
new; it can be found in a different form in work of Poonen--Schaefer--Stoll
\cite{PSS}. However, what makes our method new is the purely algebraic approach
to the problem, which is available in rather greater generality.

Before starting, we modify our model slightly and instead consider the
\defi{rational $S_3$ model} from Elkies \cite{Elkies}, which is given by
\begin{equation}\label{eq:KleinS3}
  X : x^4 + y^4 + z^4 + 6 (x y^3 + y z^3 + z x^3) - 3 (x^2 y^2 + y^2 z^2 + z^2
  x^2) + 3 x y z (x + y + z) = 0 .
\end{equation}
Consider the field $K = \Q (i, \theta)$ where $\theta^4=-7$. Over $K$, the
curve $X$ admits the involution defined by the matrix
\begin{equation}\label{eq:KleinS3invo}
  \left(\begin{array}{ccc}
    2 & -3 & -6 \\
    -3 & -6 & 2 \\
    -6 & 2 & -3
  \end{array}\right) \in \PGL_3(\Q) = \Aut(\PP^2)(\Q)
\end{equation}
This involution has the fixed point
\begin{equation}
  P = ( (i + 1) \theta^3 + 12 i \theta^2 -21 (i-1) \theta + 66 : 3(i + 1)
  \theta^3 + 36i\theta^2 -63 (i-1) \theta - 50 : 124 ) .
\end{equation}
Using branches, we will construct an isomorphism between the curve $(X;P)$ and
its conjugates that satisfies the cocycle relation.

Let $\sigma$ be the automorphism of $K$ such that $\sigma (\theta) = i\theta$
and $\sigma (i)=i$; and similarly let $\tau(\theta) = \theta$ and $\tau (i) =
-i$. Then $\sigma$ and $\tau$ satisfy the standard relations for the dihedral
group of order $8$, the Galois group of the extension $K$.

One surprise here is that it is impossible to construct a Weil cocycle for the
full extension $K \ext \Q$. In fact there exist isomorphisms $\phi_{\sigma} :
(\lexp{\sigma}{X};\lexp{\sigma}{P})$ and $\phi_{\tau} : (\lexp{\tau}{X};
\lexp{\tau}{P})$ but no matter which we take, the automorphism
\begin{equation}\label{eq:aprod}
  \phi_{\sigma}
  \lexp{\sigma}{\phi_{\sigma}}
  \lexp{\sigma^2}{\phi_{\sigma}}
  \lexp{\sigma^3}{\phi_{\sigma}}
\end{equation}
of $X$ that corresponds to the cocycle relation for $\sigma$ is never trivial,
but always of order $2$. This strongly suggests that we should find an
extension $L$ of $\Q$ with Galois group dihedral of order $16$ into which $K$
embeds. Such extensions do indeed exist; one can be found in the ray class
field of $\Q(i)$ of conductor $42$, obtained as the splitting field of the
polynomial $63 t^8 - 70 t^4 - 9$ over $\Q$.

The above extension is exactly the one that shows up when considering the
branches of the \Belyi\ map $q : X \to \Aut (X) \backslash X$, which can be
calculated by using the methods in Poonen--Schaefer--Stoll \cite{PSS}. Note
that it would have sufficed for our purposes to quotient out by the involution
fixing $P$; however, we use the \Belyi\ map to show that it, too, can be used
for purposes of descent.

We use $x = s$ as a local parameter at the point $P$, and modify the coordinate
on $\Aut (X) \backslash X$ such that $P$ is sent to $0$. After dehomogenizing
by putting $z = 1$ and determining $y = y(s)$ as a power series in $s$, we
obtain a power series expression in $s$ for $q$ by composition, that is, by
evaluating the \Belyi\ map $q$, considered as an explicit function $q (x,y)$ of
$x$ and $y$, in $( x(s), y(s) ) = (s, y(s))$. This power series looks like
\begin{equation}
  s \to c_2 s^2 + O(s^3)
\end{equation}
The leading coefficient $c_2$ of this power series has minimal polynomial
\begin{equation}\label{eq:branchpol}
  27889 t^4 - 1869588 t^3 - 18805122 t^2 + 1869795900 t - 25658183943
\end{equation}
over the rationals. Since the branch is an inverse of the power series above
under composition, it will look like
\begin{equation}
  s \to (1/\sqrt{c_2}) s^{1/2} + O(s) .
\end{equation}
The field generated by its leading coefficient can therefore be obtained as the
splitting field of the polynomial (\ref{eq:branchpol}) evaluated at $t^2$. This
is indeed the extension $L$ constructed above: we have
\begin{equation}
  \frac{1}{\sqrt{c_2}} = \frac{1}{16032}(211239\eta^7 - 66339\eta^5 -
  163835\eta^3 + 98343\eta)
\end{equation}
where $\eta$ is a root of $63 t^8 - 70 t^4 - 9$.

We now evaluate all Puiseux series for $x,y,z$ involved at the branch
$(1/\sqrt{c_2}) s^{1/2} + O(s)$ instead of $s$. This allows us to study the
Galois action; by dehomogenizing and comparing the conjugated Puiseux series
$(\lexp{\sigma}{X}(s), \lexp{\sigma}{y}(s))$ with the fractional linear
transformation of $(x,y)$ under the automorphisms of $X$, we can read off which
$\alpha_{\sigma}$ maps the former branch to the latter.

There exist elements $\widetilde{\sigma},\widetilde{\tau}$ satisfying the
relations for the dihedral group of order $16$, generating the Galois group of
$L$, and restricting to $\sigma,\tau$ on $K$. For these elements, the conjugate
branch $(\widetilde{\sigma}(x)(s), \widetilde{\sigma}(y)(s))$, respectively
$(\widetilde{\tau}(x)(s), \widetilde{\tau}(y)(s))$, is sent to $(x(s),y(s))$ by
the ambient transformation
\begin{equation}
  \alpha_{\widetilde{\sigma}} =
  \left(\begin{array}{ccc}
            -6 & -2 \theta^2 + 2 &   \theta^2 + 11 \\
     2 \theta^2 + 2 &       -10  & 3 \theta^2 + 1  \\
     -\theta^2 + 11 & -3 \theta^2 + 1 &          2
  \end{array}\right)  ,
\end{equation}
respectively
\begin{equation}
  \alpha_{\widetilde{\tau}} =
  \left(\begin{array}{ccc}
             -6 & 2 \theta^2 + 2 &  -\theta^2 + 11 \\
     -2 \theta^2 + 2 &       -10 & -3 \theta^2 + 1 \\
       \theta^2 + 11 & 3 \theta^2 + 1 &          2
  \end{array}\right)  .
\end{equation}
In particular $\sigma (P)$, respectively $\tau (P)$, is sent to $P$ by
$\alpha_{\widetilde{\sigma}}$, respectively $\alpha_{\widetilde{\tau}}$. By
uniqueness of $\alpha_{\widetilde{\sigma}}$ and $\alpha_{\widetilde{\tau}}$
when considering its action on the branches, we do get a Weil cocycle this
time. Moreover, the corresponding descent leads to a rational branch of the
\Belyi\ map $q$. After calculating a coboundary and polishing the result, we
get the model
\begin{equation}
  X_0 : x^4 + 2 x^3 y + 3 x^2 y^2 + 2 x y^3 + 18 x y z^2 + 9 y^2 z^2 - 9 z^4 =
  0
\end{equation}
as in Poonen--Schaefer--Stoll \cite{PSS}. The curve $X_0$ admits the point
$(0:1:0)$, which is a rational branch for the quotient map $(x:y:z) \mapsto
(x:y:z^2)$ to the curve
\begin{equation}
  x^4 + 2 x^3 y + 3 x^2 y^2 + 2 x y^3 + 18 x y z + 9 y^2 z - 9 z^2 = 0
\end{equation}
in $\PP(1,1,2)$.

We do not give explicit isomorphisms with the models \eqref{eq:KleinClassic}
and \eqref{eq:KleinS3} here, as they are slightly unwieldy to write down; a
matrix $\alpha_0$ inducing an isomorphism $X \xrightarrow{\sim} X_0$ can be
found by solving the equations
\begin{equation}\label{eq:WeilCobMat}
  \begin{split}
    \alpha_0 \alpha_{\sigma} & = \lexp{\sigma}{\alpha_0} \\
    \alpha_0 \alpha_{\tau}   & = \lexp{\sigma}{\alpha_0}
  \end{split}
\end{equation}
up to a scalar, where $\alpha_\sigma,\alpha_\tau \in \PGL_3(L)=\Aut(\PP^2)(L)$
induce the isomorphisms $\sigma,\tau$, respectively.  The equations
\eqref{eq:WeilCobMat} are none other than \eqref{eq:WeilCob}, and are the
explicit version of Step $5$ in Method \ref{algm:computedescent}. Once we lift
the $\PGL_3 (L)$-cocycle $\sigma \mapsto \alpha_{\sigma}$ to $\GL_3 (L)$, then
we can forget about the scalar factor, so that \eqref{eq:WeilCobMat} can be
solved by expressing the entries of $\alpha_0$ as combinations of a $\Q$-basis
of $L$. This in turn reduces solving \eqref{eq:WeilCobMat} to solving a system
of linear equations; it then remains to find a solution of relatively small
height to find a reasonable descent morphism.

\begin{rmk}
  Lifting the cocycle $\sigma \mapsto \alpha_{\sigma}$ mentioned above to
  $\GL_3 (L)$ is not trivial; for general plane curves it requires modifying an
  arbitrary set-theoretic lift by suitable scalars, which in turn comes down to
  determining the rational points on a certain conic. For plane quartics,
  however, things are slightly easier, since we can consider the morphism on
  the space of differentials that our cocycle induces.
\end{rmk}

The case of a ramification point of index $3$ is considerably easier and yields
the model
\begin{equation}
  7 x^3 z + 3 x^2 y^2 - 3 x y z^2 + y^3 z - z^4 = 0 ,
\end{equation}
on which the points $(1:0:0)$ and $(0:1:0)$ are fixed for the automorphism
$(x:y:z) \mapsto (\omega^2 x: \omega y: z)$ where $\omega$ is a primitive cube
root of unity. In this case, we do not insist on the branch being rational in
order to get a nicer model for the marked curve.

Finally, the original model \eqref{eq:KleinClassic} already admits the rational
point $(0:0:1)$, which is stable under the automorphism $(x:y:z) \to (\zeta_7^4
x : \zeta_7^2 y : \zeta_7 z)$ of order $7$.

\begin{rmk}
  Yet another way to find the descent above is to diagonalize the automorphism
  \ref{eq:KleinS3invo} to the standard diagonal matrix with diagonal
  $(1,1,-1)$, which can be done by a $\Q$-rational transformation. The $4$
  fixed points of the involution are then on the line $z = 0$. By transforming
  the coordinates $x$ and $y$ over $\Qbar$ we can ensure that the involution
  remains given by the diagonal matrix $(1,1,-1)$ and the set $4$ points
  remains still defined over $\Q$, while also putting one of these points at
  $(0:1:0)$. This normalization reduces the automorphism group sufficiently to
  make the rest of the descent straightforward.
\end{rmk}

\subsection{Wild branches and wild descent}

Determining branches in the wild case is less intuitive than in the tame case.
It turns out that branches can still be represented by certain power series,
but this time the denominators of the exponents involved are unbounded---so
although the branch belongs to the separable closure of the power series ring,
it cannot be represented by a Puiseux series. We instead have to take recourse
to \defi{generalized power series} \cite{Kedlaya}, in which the denominators of
the exponents form an unbounded set. We illustrate what this involves by
considering two examples.

\begin{exm}\label{exm:genpow}
We start off with the standard Artin--Schreier extension
\begin{equation}\label{eq:ASClassical}
  y^p - y = x .
\end{equation}
The projection $f : (y, x) \to x$ is then unramified outside $\infty$, and in
fact this is a Galois cover, with automorphism group generated by $(x, y)
\mapsto (x, y + 1)$. Transforming coordinates to interchange $0$ and $\infty$,
the extension \eqref{eq:ASClassical} becomes
\begin{equation}\label{eq:ASClassicalMod}
  y^p - y z^{p - 1} = z^{p - 1} .
\end{equation}
Now the map $f$ has $(y, z) \mapsto z$ is a Galois cover, unramified outside
$0$. Its automorphism group is isomorphic to $\Z / p \Z$, and a generator is
given by $(y,z) \mapsto (y + z, z)$.

We want to find a solution to \eqref{eq:ASClassicalMod} in a generalized power
series (with positive exponents). For this, we first set $y = c z^e$. This
gives the equation
\begin{equation}
  c^p z^{p e} - c z^{p - 1 + e} = z^{p - 1} .
\end{equation}
To get cancellation with $z^{p - 1}$, the smallest among the exponents $p e$
and $p - 1 + e$ has to equal $p - 1$. Now since $e > 0$ we can never have that
the latter exponent does the job. We get that $e = 1 - p^{-1}$ and $c^p = 1$,
so we have our leading monomial $z^{1-p^{-1}}$ of $y$. Writing $y =
z^{1-p^{-1}} + c z^e$ and continuing iteratively, we get rational exponents of
$z$ with arbitrarily large powers of $p$ in the denominator. More precisely, we
obtain the generalized power series
\begin{equation}\label{eq:ASCBranch1}
  y = \sum_{n = 1}^{\infty} z^{1 - p^{-n}};
\end{equation}
to be precise, we think of \eqref{eq:ASCBranch1} as encoding a sequence given
by its partial sums, each of which belongs to a ring $k(t^{1/p^n})$ for $n \in
\Z_{\geq 0}$.  These partial sums ``converge'' to a root of
\eqref{eq:ASClassicalMod}, but only in a formal sense.

The exponents $p e$ and $p - 1 + e$ cancel mutually precisely when $e = 1$,
confirming Proposition \ref{prop:constandfaithful} and recovering exactly the
images of \eqref{eq:ASCBranch1} under the automorphism group, to wit
\begin{equation}\label{ASCBranchw}
  y = c z + \sum_{n = 1}^{\infty} z^{1 - p^{-n}}  ,  \; c \in \F_p .
\end{equation}
We see how the wildness of the ramification leads to the denominators of the
exponents $1 - p^{-n}$ being unbounded.
\end{exm}

\begin{exm}
As a second illustration, we use a family of hyperelliptic curves in
characteristic $2$ considered by Igusa \cite{Igusa}, namely
\begin{equation}\label{eq:IgusaFam}
  y^2 - y = x^3 + a x + b x^{-1} .
\end{equation}
This curve has a distinguished point at infinity. Igusa \cite{Igusa} proved
that the family \eqref{eq:IgusaFam} contains duplicates; more precisely,
modifying $(a , b) \mapsto (\zeta_3 a, \zeta_3^{-1} b)$ does not change the
isomorphism class. The corresponding invariant subfield of $\F_2 (a,b)$ is
$\F_2 (a^3,a b)$, and we will use the marked point at infinity to descend to
this subfield.

Transforming projectively as before, we get the equation
\begin{equation}\label{eq:IgusaFamMod}
  X : x - x z = (b + a x^2 + x^4) z^2 .
\end{equation}
The hyperelliptic quotient map is given by $(x,z) \mapsto x$, and the
automorphism group (of either the curve or the marked curve) is generated by
$(x,z) \mapsto (x,z/(z + 1))$. Determining a branch, we obtain one generalized
power series that starts as
\begin{equation}
  z = b^{-1/2} x^{1/2} + b^{-3/4} x^{3/4} + \ldots
\end{equation}
%this is again just a formal expression, and by ``starts'' we mean that we have
%written a partial sum (as in the previous example) containing the terms whose
%exponent has denominator at most $4$.

Now note that from \eqref{eq:IgusaFamMod} it follows that the second branch can
be obtained by replacing $z$ by $z + c$, where
\begin{equation}
  c = \frac{x}{b + a x^2 + x^4} = b^{-1} x + a b^{-2} x^3
  + \ldots
\end{equation}
Regardless, note that if we let $\sigma$ be the automorphism sending $(a,b)$ to
$(\zeta_3 a, \zeta_3^{-1} b)$, there is an obvious isomorphism $\phi_{\sigma} :
X^{\sigma} \to X$ given by $(x,z) \mapsto (\zeta_3 x, z)$. There is a second
isomorphism $(x,z) \to (\zeta_3 x, z / (z + 1))$.  However, the first
isomorphism respects the branches, as can indeed be seen by looking at the
exponents in the corresponding expansions alone; since only the second branch
above contains a linear term in $x$, only $\phi_{\sigma}$ can send it to its
conjugate.

Taking a corresponding coboundary corresponding to the cocycle generated by
$\sigma$ is easy and comes down to substituting $b x$ for $x$ in
\eqref{eq:IgusaFamMod}. We get a descent
\begin{equation}
  x - x z =  (1 + a b x^2 + b^3 x^4) z^2
\end{equation}
in which the rational branch is clearly visible. Transforming back to Igusa's
form, we get the representative family
\begin{equation}
  y^2 - y = b^3 x^3 + a b x + x^{-1} .
\end{equation}

This family can also be obtained directly; our purpose in this example was to
illustrate how branches can be used explicitly as a tool for descent.
Descending other families of hyperelliptic curves in this way is a topic for
future work.
\end{exm}

\section{Counterexamples}\label{sec:counter}

In this final section, we prove Theorem B and consider two (counter)examples
that show that for general (non-singular) marked curves the conclusion of
Theorem \ref{thm:BirchCurve} is false. That is, we exhibit pointed curves (and
associated pointed \Belyi\ maps) that are defined over $\C$ and have field of
moduli $\R$ with respect to the extension $\C \ext \R$, yet do not descend to
$\R$.  In the first example, neither the singular curve nor its normalization
descends to $\R$. In the second example, the curve itself is defined over $\R$
but the marked curve is not, so rigidifying by marking a point creates a
descent problem where previously none existed.

\subsection{First example: the curve does not descend}

We will employ the criterion \eqref{eq:WeilCocR} of Weil descent with respect
to the extension $\C \ext \R$.

\begin{lem} \label{lem:descent_notcurve}
  Let $W$ be a variety over $\R$ and let $\rho \in \Aut(W)(\R)$ have order $4$.
  Suppose that $X \subseteq W$ is a curve over $\C$ such that $\rho
  (\overline{X}) = X$ and $\Aut(X)(\C) = \langle \rho^2|_X \rangle \simeq
  \Z/2\Z$.  Suppose $P \in W(\R) \cap X(\C)$ has $\rho(P) = P$.  Then both $X$
  and $(X;P)$ have field of moduli $\R$ with respect to the extension $\C \mid
  \R$ but neither descends to $\R$.
\end{lem}

\begin{proof}
  By hypothesis, the map $\phi = \rho|_{\overline{X}}$ gives an isomorphism
  $\phi : \overline{X} \xrightarrow{\sim} X$, and $\phi(P) = \rho(\overline{P})
  = \overline{P} = P$ since $P \in W(\R)$ and $\rho$ is defined over $\R$.  We
  have $\overline{\phi} \phi = (\rho^2)|_{\overline{X}} \neq 1$ on both $W$ and
  $X$. Therefore $\phi$ does not give rise to a descent of $X$ to $\R$. Neither
  does $\phi' = \rho^3|_{\overline{X}}$, since again it similarly satisfies
  $\overline{\phi}'\phi' = (\rho^2)|_{\overline{X}} \neq 1$.  The maps
  $\phi,\phi'$ are the only isomorphisms $\overline{X} \to X$ since any two
  would differ by an element of $\Aut(X)(\C)=\{1,\rho^2|_X\}$, so the result
  follows by the necessity part of the Weil coycle criterion (Theorem
  \ref{thm:Weil}).
%   To show that $(X;P)$ does not descend to $\R$, it is enough to show that
%   $X$ does not descend to $\R$.  Suppose that $X_0$ is a descent; then we
%   would have an isomorphism $\phi_0 : X \xrightarrow{\sim} (X_0)_\C$. Let
%   $\omega = \phi_0^{-1} \overline{\phi_0}$, defined via the composition
%   \begin{equation} \overline{X} \xrightarrow{\overline{\phi_0}}
%   \overline{(X_0)_\C} = (X_0)_\C \xrightarrow{\phi_0^{-1}} X. \end{equation}
%   Then the cocycle condition $\overline{\omega} \omega = 1$ is satisfied,
%   since $\overline{\omega} = \overline{\phi_0}^{-1} \phi_0$.  But we also
%   have $\rho|_{\overline{X}}^{-1} \omega \in \Aut(X)(\C)$, and by hypothesis,
%   we have $\Aut(X)(\C) = \langle \rho^2|_{\overline{X}} \rangle \simeq
%   \Z/2\Z$; so either $\omega = \rho|_{\overline{X}} = \phi$ or $\omega =
%   \rho^3|_{\overline{X}} = \phi'$, and in either case we obtain a
%   contradiction from the first paragraph.
\end{proof}

%\begin{rmk}
%  It follows from Theorem \ref{thm:BirchCurve} that a curve $X$ satisfying the
%  hypotheses of Lemma \ref{lem:descent_notcurve} cannot be nice. The theory
%  of branches developed below, and more precisely the argumentation in
%  Propositions \ref{prop:const} and \ref{prop:faithful}, will show that a Weil
%  cocycle can be extracted as long as the point $P$ is smooth.
%\end{rmk}

To give an explicit example of Lemma \ref{lem:descent_notcurve}, we take
$W=\PP^2$ and the automorphism
\begin{align}
  \begin{split}
    \rho:\PP^2 &\to \PP^2 \\
    (x:y:z) &\to (y:-x:z).
  \end{split}
\end{align}

\begin{lem} \label{lem:rhofixed1}
  We have
  \begin{equation}
    \{P \in \PP^2(\C) : \rho(P)=P \} = \{ (0:0:1) , (\pm i:1:0) \}.
  \end{equation}
\end{lem}

\begin{proof}
  In the affine open where $z = 1$ we get the equations $y = x = -x$, whose
  unique solution is given by $x = y = 0$. On the other hand, when $z = 0$ we
  find $(x : y) = (y : -x) \in \PP^1(\C)$. This implies $x = iy$, and the
  result follows.
\end{proof}

By Lemma \ref{lem:rhofixed1}, the only point $P \in \PP^2(\R)$ with $\rho(P)=P$
is $P=(0:0:1)$.  Our curve $X \subseteq \PP^2$ is a projective plane curve,
defined by a homogeneous polynomial $h(x,y,z) \in \R[x,y,z]$ with $h(0:0:1)=0$.
The condition that $\rho(\overline{X}) = X$ is equivalent to the condition that
$h(y:-x:z)$ is a scalar multiple of $\overline{h}(x:y:z)$; for simplicity, we
assume that
\begin{equation*}
  h(y:-x:z)=\overline{h}(x:y:z).
\end{equation*}
The condition that $\rho^2|_X \in \Aut(X)(\C)$ implies that
\begin{equation*}
  h(x:y:-z)=h(x:y:z).
\end{equation*}

We take $\deg h=4$ so that $\Aut(X)(\C)$ is finite.  The above conditions then
imply that
\begin{equation} \label{eqn:hxyzdef}
  h(x,y,z) = ax^4 + \overline{a}y^4 + (bx^2 - \overline{b}y^2) x y + (cx^2 +
  \overline{c}y^2)z^2 + rx^2y^2 + sixyz^2
\end{equation}
with $a,b,c \in \C$ and $r,s \in \R$.  We observe that $P = (0:0:1)$ is a
singular point of $X$ and $P$ is a nodal double point if and only if $s^2 \neq
-4|c|^2$.

To apply Lemma \ref{lem:descent_notcurve}, we need to select coefficients of
$h$ such that $X$ has no automorphisms beyond $\rho^2|_X$; we will show this is
true for a particular choice, which will then in fact imply that the resulting
computations are true for a general choice, i.e., for the curve defined by
equation \eqref{eqn:hxyzdef} over $\C[a,b,c,r,s,\Delta^{-1}]$ where $\Delta$ is
the discriminant of $h$.  We consider the curve with
\begin{equation*}
  (a,b,c,r,s)=(0,i,1+i,1,2)
\end{equation*}
so that $X$ is described by the homogeneous equation
\begin{equation} \label{eqn:hex1_eq}
  X : i (x^2 + y^2) x y + ((1 + i) x^2 + (1 - i) y^2) z^2 + x^2 y^2 + 2 i x y
  z^2 = 0 .
\end{equation}
Standard techniques in computational algebraic geometry (we performed our
computation in \textsc{Magma} \cite{Magma}) show that $P$ is the only singular
point of $X$, so that $X$ has geometric genus $2$. We verify that $\Aut(X)(\C)
\simeq \Z/2\Z$ by computing that the normalization of $X$ over $\R$ is given by
\begin{equation}\label{eqn:xex1genus2cv}
  Y : y^2 = (i+1)x^5 + (-i - 1)x^4 + 4x^3 + (-i + 1)x^2 + (-i + 1)x
\end{equation}
and by a computation using invariant theory (again a computation in
\textsc{Magma} \cite{ANTS}) we see that indeed $Y$ has automorphism group $\Z /
2 \Z$ generated by the hyperelliptic involution. Since every automorphism group
of a singular curve lifts to its normalization, we see that indeed $\Aut (X)
(\C) \simeq \Z / 2 \Z$. Lemma \ref{lem:descent_notcurve} therefore applies to
show that \ref{eqn:hex1_eq} is an explicit counterexample.

Alternatively, one can compute that the Igusa--Clebsch invariants of $Y$ are
defined over $\R$, so that the field of moduli is $\R$, but that there is an
obstruction to the curve $Y$ being defined over $\R$, as described by Mestre
\cite{Mestre}. On the other hand, one shows that the existence of a descent of
$Y$ is equivalent with that for $X$ since $P$ is the unique singular point of
$X$.

%\begin{rmk}
%  In this construction, the canonical model over $\R$ of $\Aut(X) \backslash X$
%  from the upcoming Definition \ref{def:canmod} has an $\R$-rational point (the
%  singular image of the singular point $P$), so the rational point criterion
%  Theorem \ref{thm:DebEms} that we will establish in the next section fails for
%  singular curves.
%\end{rmk}

\begin{rmk}\label{rmk:ShimEarle}
  With the equation \eqref{eqn:xex1genus2cv} in hand, one sees how to recover
  this class of examples in a different way, as follows.  Let $W_0$ be the
  conic defined by $x^2+y^2+z^2=0$ in $\PP^2$ over $\R$.  Then $W_0(\R) =
  \emptyset$.  Let $S \subseteq W_0(\C)$ be a subset of $6$ distinct points
  over $\C$ forming $3$ complex conjugate pairs.  Then there exists a smooth
  hyperelliptic curve $X'$ over $\C$ branched over the set $S$ that does not
  descend to $\R$: this follows from work of Mestre \cite{Mestre}, but it is
  also a rephrasing of the results by Shimura and Earle reviewed by Huggins
  \cite[Chapter 5]{Huggins}.

  Now let $Q_1,Q_2 \in S$ be such that $Q_2 = \overline{Q_1}$ and let $P_1,P_2
  \in X'(\C)$ be their ramified preimages. By the cocycle condition, the points
  $\overline{P_1}, \overline{P_2} \in \overline{X'} (\C)$ are mapped to
  $P_1,P_2 \in X' (\C)$ under the isomorphism between $\overline{X'}$ and $X'$.
  Now pinch together the points $P_1,P_2$ to obtain a curve $X$ over $\C$ with
  a double point $P$.  Then $\overline{P} \in \overline{X} (\C)$ is mapped to
  its conjugate $P$ on $X$ (``itself'' on $W_0$) by construction, and so
  $(X;P)$ also has field of moduli $\R$.
\end{rmk}

To obtain a map $f:X \to \PP^1$ unramified away from $\{0,1,\infty\}$, thereby
completing the first proof of Theorem
%\ref{thm:BirchSingular
B, we first choose a function $\pi$ on $W$ defined over $\R$ that is invariant
under $\rho$ and that is nonconstant on $X$.  This yields a morphism $\pi|_X:X
\to \PP^1$, and because $\rho$ maps $X$ to $\overline{X}$ and $\pi$ is
invariant under $\rho$, the branch locus of $\pi|_X$ is invariant under complex
conjugation.

In the specific example \eqref{eqn:hex1_eq} above, we take $\pi = x^2 + y^2$
and compute that the branch locus is described by the vanishing of the
homogeneous polynomial
\begin{equation}
  3072u^7v + 4352u^6v^2 + 5840u^5v^3 + 3424u^4v^4 + 920u^3v^5 + 104u^2v^6 +
  5uv^7,
\end{equation}
which shows that it is in fact even stable under $\Gal(\Qbar \ext \Q)$.

We can now apply the following slight strengthening of the result of \Belyi.

\begin{thm}[\Belyi\ \cite{Belyi2}]\label{thm:BelyiStrong}
  Let $g : X \to \PP^1$ be a map of curves over $\Qbar$ whose branch locus in
  $\PP^1(\Qbar)$ is defined over a number field $F \subset \Qbar$. Then there
  exists a morphism $\alpha : \PP^1 \to \PP^1$ defined over $F$ such that $f =
  \alpha \circ g$ is a \Belyi\ map.
\end{thm}

Theorem \ref{thm:BelyiStrong} applies equally well to the extension $\C \ext
\R$, essentially since every \Belyi\ map is already defined over $\Qbar$. We
apply the above for $g = \pi$, whose branch locus is $\R$-rational, and let $f
= \alpha \pi : X \to \PP^1$ be the map thus obtained, unramified away from
$\{0,1,\infty\}$. Now $f$ is $\rho$-invariant since $\pi$ is, and because we
have $\overline{\pi} = \pi \rho = \pi$ on $W$ we also have
\begin{equation}
  \overline{f} \rho = \overline{\alpha \pi} \rho = \alpha \overline{\pi} \rho =
  \alpha \pi \rho^2 = \alpha \pi = f.
\end{equation}

So indeed the field of moduli of the pointed map $(X;f:X \to \PP^1; P;
0,1,\infty)$ with respect to the extension $\C \ext \R$ equals $\R$. However,
because $X$ does not descend to $\R$, neither does $(X;f;P;0,1,\infty)$.

\subsection{Second example: the curve descends}

In this second example, the curve itself descends but the marked curve does
not. The setup is described by the following lemma.

\begin{lem}\label{lem:rigprob}
  Let $X$ be a curve over $\R$ such that
  \begin{equation}
    \Aut (X) (\C) = \Aut (X) (\R) = \langle \rho \rangle \simeq \Z/4\Z  .
  \end{equation}
  Suppose that $P \in X(\C) \setminus X(\R)$ satisfies $\overline{P} =
  \rho(P)$. Then $(X_{\C} ; P)$ has field of moduli $\R$ with respect to the
  extension $\C \!\mid\! \R$ but does not descend to $\R$.
\end{lem}

\begin{proof}
  The field of moduli claim follows by taking $\phi = \rho$ to be the identity
  map, since $\overline{X} = X$ and $\rho (\overline{P}) = P$ as $\rho$ is
  defined over $\R$. Moreover, we have
  \begin{equation} \label{eqn:33}
    P = \overline{\overline{P}} = \overline{\rho(P)} = \rho(\overline{P}) =
    \rho^2(P)  .
  \end{equation}
  As in the proof of Lemma \ref{lem:descent_notcurve}, we show that no element
  $\omega \in \Aut(X)(\C) = \langle \rho \rangle$ gives a descent.  The cocycle
  condition implies $\omega \overline{\omega} = \omega^2 = 1$, so $\omega$ is
  an involution.  And yet by \eqref{eqn:33} and the assumption that
  $\overline{P} \neq P$ implies that neither the identity nor $\rho^2$ gives a
  descent of $(X;P)$, since they do not give an isomorphism to
  $(X;\overline{P})$.
%  It remains to show that $(X;P)$ does not descend to $\R$. For purposes of
%  contradiction, suppose that $(X_0 ; P_0)$ is a descent. Then we have an
%  isomorphism $\phi_0: X_\C \xrightarrow{\sim} (X_0)_\C$ such that
%  $\phi_0 (P) = P_0$ satisfies $\overline{P_0} = P_0$. The composition
%  $\phi_0^{-1} \overline{\phi_0} = \omega$ gives a map
%  \begin{equation}
%    X_\C = \overline{X}_\C \xrightarrow{\overline{\phi_0}}
%    \overline{(X_0)_\C} = (X_0)_\C \xrightarrow{\phi_0^{-1}} X_\C
%  \end{equation}
%  so $\omega \in \Aut(X) (\C)$, and since $\overline{\omega} =
%  \overline{\phi_0}^{-1} \phi_0$ we have $\omega \overline{\omega} = 1$. But
%  since $\Aut(X)(\C) = \Aut(X)(\R)$, we have that $\omega = \overline{\omega}$.
%  Therefore the cocycle condition implies that $\omega$ is an involution.
% But then $\overline{P} \neq P$; and we cannot have $\omega=\rho^2$
% If $\omega = 1$ then $\overline{\phi_0} = \phi_0$ is defined over $\R$ so
%  that under this isomorphism
%  \begin{equation}
%    \phi_0 (\overline{P}) = \overline{\phi_0} (\overline{P}) =
%    \overline{\phi_0 (P)} = \overline{P_0} = P_0 = \phi_0 (P)  ,
%  \end{equation}
%  a contradiction with our hypothesis $P \in X(\C) \setminus X(\R)$.  So we
%  must have $\overline{\phi_0} = \phi_0 \rho^2 $. But then by
%  \eqref{eqn:33} and rationality of $\rho$ we would again have
%  \begin{equation}
%    \phi_0 (\overline{P}) = \phi_0 (\overline{\rho^2 (P)}) = \phi_0
%    (\overline{\rho^2} (\overline{P})) = \phi (\rho^2 (\overline{P})) =
%    \overline{\phi_0} (\overline{P}) = \overline{\phi_0 (P)} = P_0 = \phi_0 (P)
%    .
%  \end{equation}
  Therefore no descent of $(X;P)$ exists.
\end{proof}

\begin{rmk} \label{rmk:nowtheresanobstruction}
  Admitting for a moment that curves as in Lemma \ref{lem:rigprob} do exist, we
  see that rigidifying by marking a point may very well lead to a descent
  problem where previously none existed: while both $(X_{\C},P)$ and \emph{a
  fortiori} $X_{\C}$ have field of moduli $\R$ with respect to the extension
  $\C \!\mid\! \R$, the obstruction vanishes for the non-marked curve $X_{\C}$,
  whereas it is non-trivial for the pair $(X_{\C},P)$. So while non-marked
  curves $X_{\C}$ may not descend, as was noted by Birch \cite{Birch}, neither
  may a given rigidification of $X_{\C}$, as we saw in the last subsection, and
  in fact it may even complicate descent matters further when the automorphism
  group remains nontrivial.
%
%  The same problem occurs when rigidifying a curve by a map. Couveignes
%  \cite[\S 4]{CouveignesCRF} shows that a descent obstruction can exist for
%  \Belyi\ maps $f : X \to \PP^1$ with $X$ of genus $0$, whereas in this case
%  $X$ itself obviously descends. Note that by Theorem \ref{thm:BirchBelyi} we
%  do get a descent after marking a point on $X$, which will in fact yield a
%  model with the property that $X$ is isomorphic to $\PP^1$ over the extension
%  of the field of moduli induced by marking the point.
\end{rmk}

To exhibit an explicit curve meeting the requirements of Lemma
\ref{lem:rigprob}, we first construct a smooth curve $Y$ defined over $\R$ with
\begin{equation*}
  \Aut(Y) (\R) = \Aut(Y) (\C) \simeq \Z/4\Z.
\end{equation*}
To do this, we choose distinct $x_1, x_2 \in \C$ and let
\begin{equation}
  p (x) = (x^2 + 1) \prod_{i = 1}^2 (x - x_i) (x - \overline{x}_i) (x + 1/x_i)
  (x + 1/\overline{x}_i) \in \R[x].
\end{equation}
We see that $p(-1/x)=p(x)/x^{10}$. One then expects that generically the
hyperelliptic curve
\begin{equation*}
  Y : y^2 = p(x)
\end{equation*}
will have automorphism group generated by $\rho(x,y)=(-1/x, y/x^5)$. It is
possible to verify \cite{ANTS} that this indeed happens for the choice $x_1 = 1
+ i , x_2 = 2 + i$, which (after scaling) yields the genus $4$ hyperelliptic
curve
\begin{equation}
  Y : y^2 = 10x^{10} - 42x^9 + 67x^8 - 36x^7 + 23x^6 + 23x^4 + 36x^3 + 67x^2 +
  42x + 10.
\end{equation}

Consider the point $Q = (1 + i,0) \in Y(\C)$. To get Lemma \ref{lem:rigprob} to
apply, we need $\overline{Q}=\rho (Q)$. Of course this cannot be the case on
$Y$ in light of Theorem \ref{thm:BirchCurve}, but we can make it true by
construction if we pinch together the points $\overline{Q}$ and $\rho(Q)$.
Moreover, if we take care to pinch $Q$ with $\rho (\overline{Q})$ similarly,
then the resulting contraction morphism $c : Y \to X$ will be defined over
$\R$. If we let $P = c(Q)$, then by construction, the pair $(X;P)$ will satisfy
the conditions of Lemma \ref{lem:rigprob}.  Indeed, any automorphism of a
singular curve lifts to its normalization, and conversely the automorphisms of
$Y$ all transfer to $X$ because of our construction of the contraction, so that
indeed
\begin{equation*}
  \Aut(X)(\R)=\Aut(X)(\C) \simeq \Aut (Y)(\C) \simeq \Z / 4 \Z ;
\end{equation*}
the stabilizer of $P$ again corresponds to the subgroup of $\Z / 4 \Z$ of order
$2$.

A contraction morphism $c$ can be constructed in many different ways, none of
which is particularly pleasing from the point of view of finding an equation.
The first approach is to find a sufficiently ample linear system and extract
the sublinear system of sections whose values at $\overline{Q}$ and $\rho (Q)$
(resp.\ $Q$ and $\rho (\overline{Q})$) coincide. This leads to ambient spaces
that are too large to give rise to attractive equations.

An alternative approach is as follows. Define
\begin{equation}
  \begin{split}
    q(x) & = (x - x_1)(x - \overline{x}_1) (x + 1/x_1)(x + 1/\overline{x}_1) =
    x^4 - x^3 + \frac{1}{2} x^2 + x + 1, \\
    r(x) & = \frac{2}{3} x^3 - 2 x^2 + x.
  \end{split}
\end{equation}
and let $c : \AA^2 \to \AA^4$ be given by
\begin{equation}
  c(x,y) = (q(x), x q(x), r(x), y) = (t,u,v,w).
\end{equation}
Given a point in the image of $Y$ under $c$, we can always read off the
original $y$-coordinate from $w$. We can also recover the $x$-coordinate $u /
t$ as long as $t \neq 0$. The latter only occurs if $x$ is a root of $q$; these
roots all have the property that $t = u = w = 0$. On the other hand, one
verifies that $r$ assumes exactly two distinct values on this set of roots;
moreover, the preimage of one of these values is given by $\{ Q,
\rho(\overline{Q}) \}$, and the other by $\{ \overline{Q}, \rho (Q) \}$. We
see that the morphism $c$ constracts the pairs $\{ Q , \rho(\overline{Q}) \}$
and $\{ \overline{Q}, \rho (Q) \}$ exactly in the way that we wanted.

This method gives only an affine open of $X$, but this open contains all the
singular points that we are interested in, so that completing the corresponding
model smoothly at infinity will give the desired model of $X$. On the patch $t
\neq 0$ we can describe the image of $c$ by the following equations:
\begin{equation}
  \begin{aligned}
    t^5 - t^4 - t^3 u - 1/2 t^2 u^2 + t u^3 - u^4 &= 0, \\
    t^3 v - t^2 u + 2 t u^2 - 2/3 u^3 &= 0, \\
    10t^{10} w^2 - 10t^{10} - 42 t^9 u - 67 t^8 u^2 - 36 t^7 u^3 - 23 t^6
      u^4 &  \\
      - 23 t^4 u^6 + 36 t^3 u^7 - 67 t^2 u^8 + 42 t u^9 - 10u^{10}
      & = 0.
  \end{aligned}
\end{equation}

Adding a few more coordinates, we can also describe the automorphism $\rho$:
using the contraction
\begin{equation}
  \begin{split}
    c (x,y) & = (q(x), x q(x), r(x), y, q(-1/x), -1/x q(-1/x), r(-1/x), y/x^5)
    \\      & = (t,u,v,w,t',u',v',w'),
  \end{split}
\end{equation}
it can be described by
\begin{equation}
  \rho (t,u,v,w,t',u',v',w') = (t',u',v',w',t,u,v,-w).
\end{equation}

In both cases, the full ideal of the image of $c$ can be recovered by using
Gr\"obner bases.

To conclude our argument; by Theorem \ref{thm:BelyiStrong}, we can find a map
$f:X \to \PP^1$ unramified outside $\{0,1,\infty\}$ such that $(X,f;P)$ has the
same automorphism group: one takes any function $\phi$ defined over $\R$ that
is invariant under $\Aut (X,P)$ and postcompose with a function $h$ following
\Belyi\ to get a map $\phi = h \phi_0$ with the same field of moduli and the
same obstruction.

\begin{rmk}
  The curve in the first counterexample had exactly one singular point---in a
  sense, it was ``marked'' anyway.  By contrast, the curve in this subsection
  has two singular points, which makes it important to keep track of the chosen
  marking.
\end{rmk}

\begin{rmk}
  There should be many more ways to construct singular curves over $\R$ for
  whom the field of moduli is no longer a field of definition after
  rigidification by a singular point.

  For example, if we take the automorphism $\rho:\PP^3 \to \PP^3$ given by the
  cyclic permutation of the coordinates $\rho(x:y:z:w)=(y:z:w:x)$ and considers
  the generic complete intersection $X$ of two polynomials invariant under
  $\rho$ containing the point $P=(1:i:1:i)$ with $\rho(P)=\overline{P}$, then
  we expect that the curve $X$ will satisfy the hypotheses above.  For example,
  if we define
  \begin{equation}
    \rho(m)=m+\rho^*(m)+(\rho^{*2})(m)+(\rho^{*3})(m)
  \end{equation}
  for $m \in k[x,y,z,w]$ a monomial, then we believe that the curve $X$ of
  genus $7$ defined by
  \begin{equation}
  \begin{gathered}
      \rho(x^2)=x^2 + y^2 + z^2 + w^2 = 0 \\
      5 \rho(x^4) + \rho(x^3y) + \rho(x^3z) -4\rho(x^2yw) -2 \rho(x^2zw) +
      3\rho(x^2y^2) + 7\rho(xyzw) = 0
  \end{gathered}
  \end{equation}
  has this property, as we can verify that $\Aut(X)(\F_p) \simeq \Z/4\Z$ for a
  number of large primes $p$; however, it is much more involved to provably
  compute the automorphism group for such a curve, which is why we have
  preferred to stick with this admittedly more special case.
\end{rmk}

%\begin{rmk}
%  The methods from this section can be used to show that given any number $n$
%  of points to be fixed, there exist tuples $(X,P_1,\dots,P_n)$ that give rise
%  to a descent obstruction. Conversely, for any given curve there exists a
%  number $n$ such that the descent obstruction for the marked curve vanishes as
%  soon as they are marked by $n$ distinct points (since the number of
%  orbifold points on the quotient $\Aut(X) \backslash X$ is finite).
%\end{rmk}

\appendix 

\section{Descent of marked Galois \Belyi\ maps in genus zero}\label{sec:XPf}

In this appendix, we show how branches can be used to provide explicit descent
of marked Galois \Belyi\ maps in genus zero.  The results are classical, but
the method to derive them is quite pleasing so we provide it here.

In the absence of unipotent automorphisms, a (geometrically generically) Galois
map ${W \to V}$ between curves of genus $0$ over $\Fbar$ is always a quotient
by one of the groups $C_n, D_n, A_4, S_4, A_5$ for $n \geq 2$. Because we want
to use the language of \Belyi\ maps, throughout this appendix we suppose
$\opchar F = 0$; the results hold equally well when $\#G$ is coprime to
$\opchar F$, but we are reluctant to use the term \Belyi\ maps in this case.
The quotients mentioned previously are then all \Belyi\ maps over $\Fbar$, and
in what follows we wish to consider some special descents of these maps. This
will expand on the results in Couveignes--Granboulan
\cite{CouveignesGranboulan}, in which forms of these \Belyi\ maps were already
given.  

\subsection{Descent of lax \Belyi\ maps}

We will slightly broaden our notion of \Belyi\ maps over $F$ in order to allow
their branch divisor to be an arbitrary $F$-rational divisor instead of merely
consisting of multiples of points individually defined over $F$, which is to
say want to allow the situation where the points on $\PP^1$ are not marked.

\begin{defn}\label{def:GenBelyi}
  A \defi{lax \Belyi\ map} over $F$ is a map $f : X \to \PP^1$ of curves
  over $F$ that is ramified above at most three points. A \defi{marked lax
  \Belyi\ map} over $F$ is a pair $(X;f;P)$, where $f:X \to \PP^1$ is a lax
  \Belyi\ map over $F$ and $P \in X(F)$ is a point of $X$ over $F$.
\end{defn}

By Birch's theorem, any marked lax \Belyi\  map descends to its field of
moduli, and the method of branches gives a descent of marked lax \Belyi\  maps. 

However, it turns out that if we consider the case where $X$ has genus $0$,
then things are much easier, essentially because we have an obvious descent
$(\PP^1;\infty)$ of $(X;P)$ that we can exploit. This leads to a useful and
easily memorizable trick that resembles our approach to hyperelliptic curves in
the previous section.

In what follows, we suppose that $f:X \to \PP^1$ is a lax \Belyi\  map with $X$
of genus $0$ such that $f$ is (geometrically generically) Galois with field of
moduli $F$ with respect to the extension $\Fbar \ext F$. We suppose $P$ is a
ramification point of $f$ of \emph{non-trivial} order $n \in \Z_{\geq 2}$,
which we suppose maps to $0 \in \PP^1$ for the sake of simplicity. Our
reasoning will now be very similar to that in section \ref{subsec:hypcase}.

After applying an automorphism if necessary, there exists a model $\PP^1 \simeq
X$ such that:
\begin{itemize}
  \item $P=\infty=(1:0) \in \PP^1$, and
  \item The automorphism group of $(X,f;P)$ is generated by $(x:z) \mapsto
    (\zeta_n x:z)$ with $\zeta_n$ a primitive $n$th root of unity.
\end{itemize}
This means that if we let $t$ be the affine coordinate $z/x$ the function $f$
can be identified with a power series expansion
\begin{equation}
  f (t) = \sum_{i = 1}^{\infty} c_i t^{i n} .
\end{equation}

Since $\Aut(X,f;P)(\Fbar) \simeq \Z/n\Z$, the set $\{ i \geq 1 : c_i \neq 0 \}$
generates $\Z$ as a group. So there exists a finite subset $I$ of this set with
this property as well. Let $c_I = ( c_i )_{i \in I}$ be the corresponding point
in the projective space weighted by its indices, and choose a normalized
representative $c_{I,0}$ of $c_I$ \cite{LR}. There exists a scaling $t \mapsto
\alpha t$ of $t$ that transforms $c_I$ into $c_{I,0}$, which is uniquely
determined up to a power of the automorphism $t \mapsto \zeta_n t$ since the
elements in $I$ generate $\Z$.

\begin{prop}\label{prop:descg0}
  With notations as above, let
  \begin{equation}
    f_0 (t) = f (\alpha t) = \sum_{i = 1}^{\infty} c_{i,0} t^{i n} .
  \end{equation}
  Then $f_0$ has coefficients in $F$ and $(X,f_0; \infty)$ is a descent of
  $(X,f;P)$.
\end{prop}

\begin{proof}
  Any isomorphism between $f_0$ and its conjugate $\lexp{\sigma}{f_0}$
  normalizes the group of automorphisms generated by $(x:z) \mapsto (\zeta_n
  x:z)$ and fixes $(1:0) \in \PP^1$. Therefore it is given by a scaling $t
  \mapsto \alpha_{\sigma} t$.  The subset $c_I$ is defined over $F$, and in
  particular it is fixed by Galois conjugation. By uniqueness of the normalized
  representative, we therefore see that $\alpha_{\sigma}$ is a power of
  $\zeta_n$, and in particular that $f_0$ is stable under conjugation by
  $\sigma$. Since $\sigma$ was arbitrary, we see that indeed $f_0$ is defined
  over $F$.
\end{proof}

\begin{rmk}
  Note that contrary to the descent obtained in the proof of Theorem
  \ref{thm:BirchExp} the descent constructed in Proposition \ref{prop:descg0}
  may not give rise to an $F$-rational branch. If desired, such a branch can be
  found by further scaling the $t$-coordinate.
\end{rmk}

\subsection{Explicit descents}

In what follows, we exhibit descents of all the Galois lax \Belyi\ maps of
genus $0$, ordered by their automorphism group $G$. As mentioned above, $G$ is
then one of the groups $C_n, D_n, A_4, S_4, A_5$. Given such a group $G$, let
$f : X \to \PP^1$ be a \Belyi\ map with Galois group $G$. Then $f$ branches
over $2$ or $3$ points. The former possibility occurs only for $G = C_n$, and
in this case we slightly abusively add one more point to the set of branch
points, over which only trivial ramification (of index $1$) occurs. By
transivity of $\Aut (\PP^1)$ on triples of points of $\PP^1$, we may assume the
branch points in $\PP^1$ to be $0, 1, \infty$. We denote the uniform
ramification indices of $f$ over these points by $e_0, e_1, e_{\infty}$. Any
two \Belyi\ maps with the same Galois group and the same values of $e_0, e_1,
e_{\infty}$ are in fact isomorphic; we choose an ordering of these values,
which therefore amounts to specifying the map $f$ up to isomorphism.

We now want to determine the corresponding descents $f_0$ of marked lax Belyi
maps. Since $G$ acts transitively on the fibers, it suffices to determine the
descents for any choice of a point $P$ over one of $0, 1, \infty$. The proof of
Proposition \ref{prop:descg0} shows that given such a point it is always
possible to find a descent $f_0$ with the same ramification over $0, 1,
\infty$. We give equations for such descents by indicating triples $p_0, p_1,
p_{\infty}$ with the property that
\begin{equation*}
  p_0 - p_{\infty} = p_1
\end{equation*}
with $p_i (t) \in F [t]$ is an $e_i$-th power. These polynomials describe a
rational function $f_0$ via
\begin{equation*}
  f_0 = \frac{p_0}{p_{\infty}} = 1 + \frac{p_1}{p_{\infty}} .
\end{equation*}
The function $f_0$ is a marked \Belyi\ map defined over $F$, hence in
particular a marked lax \Belyi\ map. We give three such triples $p_0, p_1,
p_{\infty}$ in general, corresponding to the choice of $0, 1, \infty$ over
which we mark a point $P$. Occasionally a descent for one of these choices also
works for others.

In general, more descents $f_0$ exist. In fact the branch point on $\PP^1$ that
is the image of the marked point $P$ will always be rational over the field of
moduli; this was also the reason that we did not consider conics as the target
space in our Definition \ref{def:GenBelyi}. However, the definition of a
\emph{lax} \Belyi\ map allows for the possibility that the other two branch
points are conjugate over $F$. This can occur if these points have the same
branch indices.

We can determine corresponding twists as follows. Suppose for simplicity that
the coinciding ramification indices are $e_0$ and $e_{\infty}$, as indeed we
may up to reordering. Then we can consider the rational function $f_0$ over $F$
with branch points $0, 1, \infty$ constructed above, for a choice of $P$ over
$1$. We can then construct the function $\sqrt{d} (f_0 + 1)/(f_0 - 1)$, which
moves the branch point $1$ to $\infty$ and the pair $0,\infty$ to $\pm
\sqrt{d}$. Normalizing as in Proposition \ref{prop:descg0}, we get twists over
$F$ for all non-zero values of $d$; the individual points of their branch
divisor are defined over the quadratic extension of $F$ determined by $d$.
Having additionally determined such equations, all twists of marked lax \Belyi\
maps can then be recovered, if so desired, by scaling in the $t$-coordinate.
Note that such a scaling changes the field of definition of the branch at the
marked point.

\subsubsection*{Case $G \simeq C_n$}
Here we can take $e_0 = e_{\infty} = n$ and $e_1 = 1$ to get the triple
\begin{equation}
  p_0 = t^n, \:
  p_{\infty} = 1, \:
  p_1 = t^n - 1 .
\end{equation}
The corresponding rational function
\begin{equation}
  f_0 = t^n = 1 + (t^n - 1)
\end{equation}
in fact has rational preimages over all three points $0, 1, \infty$, so that
the given triple works for all markings.

Twisting this map to get it to ramify over $\pm \sqrt{d}$ is actually not so
straightforward as the procedure above, essentially because there is no
branching over $1$, so that we cannot apply our normalization argument.
However, a corresponding lax \Belyi\  map was constructed by
Lercier--Ritzenthaler--Sijsling \cite[Proposition 3.16]{LRS}. 
% We pass over this case here since the corresponding points have trivial
% ramification index, so that it is slightly abusive to include them, as was
% mentioned above.

\subsubsection*{Case $G \simeq D_n$}
Taking $e_0 = e_{\infty} = 2$ and $e_1 = n$ we can use the triples
\begin{equation*}
  p_0 = (t^n + 1)^2 , \:
  p_{\infty} = (t^n - 1)^2 , \:
  p_1 = -4 t^n
\end{equation*}
with corresponding rational function
\begin{equation}\label{eq:Dn}
  f_0 = \frac{(t^n + 1)^2}{(t^n - 1)^2} = 1 + \frac{4 t^n}{(t^n - 1)^2}  .
\end{equation}
This gives marked descents over $\infty$ and $1$ since over these points $f$
has the rational preimages $t = 1$ respectively $t = 0, \infty$.  Exchanging
the roles of $0$ and $\infty$ gives a marked descent over $0$. Twisting gives
the functions
\begin{equation}
  f_0 = \frac{d t^{2n} + 1}{2 t^{n}}
\end{equation}
which indeed branch of index $n$ over $\infty$ and moreover of index $2$ over
$\pm \sqrt{d}$ since
\begin{equation}
  (d t^{2n} + 1) \pm 2 \sqrt{d} t^n = (\pm \sqrt{d} t^n + 1)^2 .
\end{equation}

\subsubsection*{Case $G \simeq A_4$}
In this case we take $e_0 = e_{\infty} = 3$ and $e_1 = 2$. We get a triple
\begin{equation}
  p_0 = \bigl( t(t^3 + 8) \bigr)^3 , \:
  p_{\infty} = 2^6 (t^3 - 1)^3 , \:
  p_1 = (t^6 - 20 t^3 - 8)^2,
\end{equation}
which gives rise to the rational points $t = 0$ over $0$ and $t = \infty$ over
$\infty$. A triple that has the rational point $t = \infty$ over $1$ is given
by
\begin{equation}
  p_0 = (3 t^4 + 6 t^2 - 1)^3 , \:
  p_{\infty} = (3 t^4 - 6 t^2 - 1)^3 , \:
  p_1 = 2^2 3^2 \bigl( t (3 t^4 + 1) \bigr)^2 .
\end{equation}
A corresponding family of twists branching over $\infty$ and $\pm \sqrt{d}$ is
given by
\begin{equation}
  f_0
  =
  \frac{d^3 t^{12} + 99 d^2 t^8 - 297 d t^4 - 27}{18 (t (d t^4 + 3))^2}
  =
  \pm \sqrt{d} + \frac{(d t^4 \mp 6 \sqrt{d} t^2 - 3)^3}{18 (t (d t^4 + 3))^2} .
\end{equation}

\subsubsection*{Case $G \simeq S_4$}
For this case, we take $e_{\infty} = 4$, $e_0 = 3$, $e_1 = 2$. Now the branch
indices are distinct, so we do not get extra twists, but only the three
following triples for rational points over $\infty, 0, 1$ respectively:
\begin{align}
  \begin{split}
    p_0 & = (t^8 + 14 t^4 + 1)^3 \\
    p_{\infty} & = 2^2 3^3 \bigl( t(t^4 - 1) \bigr)^4, \\
    p_1 & = (t^{12} - 33 t^8 - 33 t^4 + 1)^2 , \\
  \end{split}
\end{align}
\begin{align}
  \begin{split}
    p_0 & = 2^8 (t^7 - 7 t^4 - 8 t)^3 \\
    p_{\infty} & = (t^6 + 20 t^3 - 8)^4 \\
    p_1 & = - (t^{12} - 88 t^9 - 704 t^3 - 64)^2 ,
  \end{split}
\end{align}
and
\begin{align}
  \begin{split}
    p_0 & = (3 t^8 + 28 t^6 - 14 t^4 + 28 t^2 + 3)^3 \\
    p_{\infty} & = 3^3 (t^6 - 5 t^4 - 5 t^2 + 1)^4 \\
    p_1 & = 2^4 (9 t^{11} + 11 t^9 + 66 t^7 - 66 t^5 - 11 t^3 - 9 t)^2 .
  \end{split}
\end{align}

\subsubsection*{Case $G \simeq A_5$}
We take $e_{\infty} = 5$, $e_0 = 3$, $e_1 = 2$. Once more the branch triples
are distinct, and the triples with rational points over $\infty, 0, 1$ are
given by
\begin{align}
  \begin{split}
    p_0 & = (t^{20} + 228 t^{15} + 494 t^{10} - 228 t^5 + 1)^3 \\
    p_{\infty} & = 2^6 3^3 \bigl( t(t^{10} - 11 t^5 - 1) \bigr)^5 \\
    p_1 & = (t^{30} - 522 t^{25} - 10005 t^{20} - 10005 t^{10} + 522 t^5 + 1)^2 ,
  \end{split}
\end{align}
\begin{align}
  \begin{split}
    p_0 & = 5^3 \bigl( t (5 t^6 + 5 t^3 + 8) (5 t^6 - 40 t^3 - 1) (40 t^6 - 5 t^3 + 1) \bigr)^3 \\
    p_{\infty} & = 2^6 (25 t^{12} + 275 t^9 - 165 t^6 - 55 t^3 + 1)^5 \\
    p_1 & = - \bigl( (5 t^6 + 1) (25 t^{12} - 1750 t^9 - 2190 t^6 + 350 t^3 + 1) \cdot \\
      & \quad\quad\quad (200 t^{12} - 500 t^9 + 2055 t^6 + 100 t^3 + 8) \bigr)^2 ,
  \end{split}
\end{align}
and finally
\begin{align}
  \begin{split}
    p_0 & = \bigl((3 t^4 + 10 t^2 + 15)(t^8 + 70 t^4 + 25)(t^8 + 60 t^6 - 370 t^4 + 300 t^2 + 25) \bigr)^3 \\
    p_{\infty} & = 3^3 \bigl((t^4 - 2 t^2 + 5)(t^8 - 20 t^6 - 210 t^4 - 100 t^2 + 25)\bigr)^5 \\
    p_1 & = 2^2 \bigl(t(t^4 - 5)(t^4 - 10 t^2 + 45)(9 t^4 - 10 t^2 + 5) \cdot \\
      & \quad\quad\quad (t^8 + 20 t^6 + 470 t^4 + 500 t^2 + 625) (5 t^8 + 20 t^6 + 94 t^4 + 20 t^2 + 5)\bigr)^2 .
  \end{split}
\end{align}

The above method extends to curves and \Belyi\ maps of genus $0$ with unipotent
automorphism group, and to higher genus (hyperelliptic or not).  We expect this
will be useful in building databases of \Belyi\ maps with small defining
equations, a topic for future work.

\end{document}